\documentclass[leqno, letterpaper, 11pt]{article}

\usepackage{amssymb}
\usepackage{amsmath}
\usepackage{amsthm}
\usepackage{tikz}
\usepackage{fullpage}
\usepackage{mathrsfs}
\usepackage{subfigure}
\usepackage{verbatim}
\usepackage{enumitem}
\usepackage{bbm}
\usepackage{siunitx}
\usepackage{relsize}
\usepackage{ytableau}
\usepackage{tikz}
\usetikzlibrary{through,calc,positioning,decorations.pathreplacing}
\usepackage{booktabs}

\graphicspath{{Figures/}}

\def\Z{\mathbb{Z}}

\newcommand{\bR}{{\mathbb R}}
\newcommand{\bZ}{{\mathbb Z}}

\newcommand{\cA}{{\mathcal A}}

\newcommand{\cZ}{{\mathcal Z}}

\newcommand{\cL}{{\mathcal L}}

\newcommand{\cN}{{\mathcal N}}

\newcommand{\cO}{\mathcal O}

\numberwithin{equation}{section}
     
\newcommand{\lep}{\displaystyle\;\mathlarger{\mathlarger{\triangleleft}}\;}
\newcommand{\gep}{\displaystyle\;\mathlarger{\mathlarger{\triangleright}}\;}
\newcommand{\ha}{\widehat a}
\newcommand{\hm}{\widehat m}
\newcommand{\bm}{\overline{m}}

\renewcommand{\r}{\rho}

\definecolor{ashgrey}{rgb}{0.7, 0.75, 0.71}
\definecolor{faint}{rgb}{0.9, 0.9, 0.9}
\definecolor{lgrey}{rgb}{0.8, 0.8, 0.8}
\definecolor{dgrey}{rgb}{0.5, 0.5, 0.5}
\newcommand{\gc}{\gamma_c}

\newtheorem{theorem}{Theorem}[section]

\newtheorem{lemma}[theorem]{Lemma}
\newtheorem{corollary}[theorem]{Corollary}
\newtheorem{prop}[theorem]{Proposition}

\newcommand{\prob}[1]{\mathbb{P}\left(#1\right)}
\newcommand{\expe}[1]{\mathbb{E}\left(#1\right)}

 \newcommand{\note}[1]{{\bf \textcolor{blue}
{[#1\marginpar{\textcolor{red}{***}}]}}}

\begin{document}

\begin{center}\Large

{\bf Two-dimensional supercritical growth dynamics with\\ one-dimensional nucleation}
\end{center}

 \begin{center}
 {\sc Daniel Blanquicett}\\
 {\rm Department of Mathematics}\\
 {\rm University of California, Davis}\\
{\rm Davis, CA 95616}\\
{\rm \tt drbt{@}math.ucdavis.edu}
 \end{center}

 \begin{center}
 {\sc Janko Gravner}\\
 {\rm Department of Mathematics}\\
 {\rm University of California, Davis}\\
{\rm Davis, CA 95616}\\
{\rm \tt gravner{@}math.ucdavis.edu}
 \end{center}

\begin{center}
{\sc David Sivakoff}\\
{\rm Departments of Statistics and Mathematics}\\
{\rm The Ohio State University}\\
{\rm Columbus, OH 43210, USA}\\
{\rm \tt dsivakoff{@}stat.osu.edu}
\end{center} 

\begin{center}
{\sc Luke Wilson}\\
{\rm Departments of Physics and Mathematics}\\
{\rm The College of Wooster}\\
{\rm Wooster, OH 44691, USA}
\end{center} 




\begin{abstract} 
We introduce a class of cellular automata growth models on the two-dimensional integer 
lattice with finite cross neighborhoods. 
These dynamics are determined by a Young diagram $\cZ$ and the radius $\r$ 
of the neighborhood, which we assume to be sufficiently large. A point becomes occupied 
if the pair of counts of currently occupied
points on the horizontal and vertical parts of the neighborhood lies outside $\cZ$.
Starting with a small density $p$ of occupied points, we focus on the 
first time $T$ at which the origin is occupied. We show that $T$ scales as a power 
of $1/p$, and identify that power, when $\cZ$ is the triangular set that 
gives threshold-$r$ bootstrap percolation, when $\cZ$ is a rectangle, 
and when it is a union of a finite rectangle and an infinite strip. We give partial results 
when $\cZ$ is a union of two finite rectangles. The distinguishing feature of these dynamics 
is nucleation of lines that grow to significant length before most of the space is covered. 
\end{abstract}

\section{Introduction}\label{sec-intro}

To quote \cite{Mar}, ``Nucleation
means a change in a physical or chemical system that begins within a small region.'' 
Nucleation is a key factor in determining self-organization properties of physical systems \cite{JD}.
Of particular interest for probabilists is homogeneous nucleation, which happens 
due to random fluctuations in a statistically homogeneous environment without preferred 
nucleation sites, such as impurities. Starting in the 1970s, simple models  were devised to study 
such processes \cite{CLR, AL}. 
In these models, a ``small region'' is one with 
a diameter much smaller than the time scale on which a new equilibrium is reached; in this sense, 
the nucleation is strictly local. In this paper, we study models in which this is violated.
Instead, the ``small region'' is now of a lower dimension than the growth environment. 
As we restrict to 
two-dimensional environments, this translates to growth models that generate extended one-dimensional 
tentacles long before the growth covers most of the space. While evidence of such growth 
in the physical literature is scarce, there is a recent discovery of efficient nucleation 
of lines in the assembly of two-dimensional molecular arrays \cite{Che}. 
 
Our model is a local version of the one we introduced in \cite{GSS}, 
a class of rules that can accommodate fairly general interaction between two 
possible directions. To ensure monotonicity, the key parameter in such rules is a Young diagram. 
The other parameter is a finite range, which ensures locality, and which we assume to be large 
enough, but otherwise plays a limited role at our level of precision.

To proceed with precise definitions, 
we call a set $\cZ\subseteq\bZ_+^2$, where $\bZ_+=\{0,1,2,\ldots\}$,  a {\it zero-set\/} if 
$(u,v)\in\cZ$ implies $[0,u]\times [0,v]\subseteq\cZ$. If $\cZ$ is finite, then it 
is equivalent to a Young diagram in the French notation; however, 
infinite zero-sets are also of interest and will still be referred to as Young diagrams. 
Sometimes we specify $\cZ$ by the {\it minimal counts\/}, which are those pairs 
$(u_0,v_0)\in \bZ_+^2\setminus\cZ$ 
for which both $(u_0-1,v_0)$ and $(u_0,v_0-1)$ are in $\cZ\cup (\bZ_+^2)^c$. The reason for this
terminology will be clear in the next paragraph.

As announced, we consider cellular automata dynamics with cross neighborhoods with {\it range\/}
$\rho\ge 0$. That is, for  $x\in \bZ^2$ we let $\cN^h_x=[-\rho,\rho]\times \{0\}$, 
$\cN^v_x=\{0\}\times[-\rho,\rho]$ and then 
the neighborhood of $x$ is 
$$
\cN_x=\cN_x^h\cup\cN_x^v.
$$
For each $t\in \bZ_+$, we let $\xi_t\subseteq \bZ^2$ denote the collection of \textit{occupied} (or \textit{active}) vertices at time $t$. Given $\xi_0$, we define $\xi_{t}$ recursively by
\begin{equation}\label{xi-rule}
    \xi_{t+1} = \xi_t \cup \{x\in \bZ^2 : (|\cN_x^h\cap \xi_t|,|\cN_x^v\cap \xi_t|)\notin \cZ\}.
\end{equation}
That is, for each unnoccupied point $x\in \bZ^2$ we compute the pair of 
counts of currently occupied points in the horizontal and vertical parts of its neighborhood, 
and then we add $x$ to the occupied set if this pair lies outside of $\cZ$. Thus we have 
{\it solidification\/}: $\xi_t\subseteq \xi_{t+1}$. 
The definition of a zero-set also ensures {\it monotonicity\/}: enlarging $\xi_0$ can only enlarge any $\xi_t$, 
$t\ge 0$. 
We will also consider such dynamics on a finite set $S$, with {\it $0$-boundary\/}, whereby we assume 
$\xi_0\subseteq S$ and only consider $x\in S$ for occupation at all times. The most useful 
such set is $B_n=[0,n-1]^2$, on which we also often 
impose periodic boundary conditions (see Section~\ref{sec-prelim}). 
We set $\xi_\infty =\cup_{t\ge 0}\xi_t$, and call a set $A\subseteq \bZ^2$ {\it inert\/} 
if $\xi_0=A$ implies $\xi_1=A$. 

Cellular automata growth dynamics with cross neighborhoods 
were introduced in \cite{HLR}, and indeed the model considered in that paper fits our definition. 
It is, however, a critical dynamics 
\cite{BSU, BDMS}, by contrast with the supercritical ones we consider here. In another direction, the recent papers
\cite{Bla1, Bla2, Bla3} study critical dynamics with $d$-dimensional 
versions of cross neighborhoods, for $d\ge 3$.

Without loss of generality, we assume that the height of the zero-set is no larger than the 
width. 
We will also assume that the height of $\cZ$ is finite and that it does not exceed $\rho$:
$\rho\ge |\cZ\cap(\{0\}\times \bZ_+)|$. When the width of $\cZ$ is finite, we also 
assume that it does not exceed $\rho$: $\rho\ge |\cZ\cap( \bZ_+\times\{0\})|$. 
(In fact, this is not a restriction: if width exceeds $\rho$, 
we can obtain the same process by replacing $\cZ$ with the new zero-set $\cZ'$  
that agrees with $\cZ$ except that the rows of $\cZ$ that exceed $\rho$ 
are made infinite in $\cZ'$.)
These constraints make our dynamics supercritical \cite{BSU, BBMS1}. 
Such dynamics are {\it voracious\/} \cite{GG} if every starting set of minimal cardinality selected
from
$$\cA=\{A\text{ finite}:\xi_0=A\text{ generates $\xi_\infty$ with }|\xi_\infty|=\infty\}
$$
results in $\xi_\infty=\bZ^2$. Nucleation properties of 
voracious dynamics are relatively transparent as they are determined by the 
minimal sets of $\cA$ \cite{GG}. 

However, unless 
$\cZ$ consists of a single point, our dynamics are {\it not\/} voracious. 
To see this, assume that $\cZ$ has height $h$, and that its width is at least $\max(h,2)$. 
If $|\xi_\infty|=\infty$, then $|\xi_0|\ge h$ (as a set of  
$h-1$ horizontal or vertical parallel lines is inert), but 
a vertical interval of $h$ sites generates 
an occupied vertical line, which is inert. Therefore, the results of \cite{GG} do not 
apply. Instead, nucleation is governed by most efficient 
configurations of occupied sites that generate lines that grow in different directions
and interact to finally produce a configuration that expands in all directions; see 
Figure~\ref{fig-sim}.
 
To study the nucleation properties, we assume that each $x\in \bZ^2$ is included in $\xi_0$ independently with probability $p\in [0,1]$, and investigate the scaling of 
$$
T = \inf\{t\ge 0 : (0,0)\in \xi_t\},
$$ 
the first time that the origin is occupied, as $p\to 0$. Besides $T$, 
another natural quantity is the {\it critical length\/} $L_c$~\cite{Mor}. We say $B_n$ is \textit{spanned} if the dynamics on $B_n$ with $0$-boundary eventually occupies every point of $B_n$. Then
$$
L_c := \inf\{n\ge 0 : \prob{B_n \text{ is spanned}}\ge1/2\}.
$$
The advantage of $T$, especially in the asymmetric cases, is 
that it imposes no symmetry restriction on the geometry of 
the domain on which nucleation events that affect occupation of the origin happen. 
 
The most studied special case of growth dynamics is known as {\it bootstrap percolation\/} \cite{CLR}
or {\it threshold growth\/} \cite{GG}. 
In our context, it is given by an integer {\it threshold\/} $r\ge 1$,
and the triangular zero-set 
$\cZ=\{(u,v):u+v\le r-1\}$.
Therefore, a site $x$ joins the occupied set whenever the number of 
currently occupied sites in $\cN_x$ is at least $r$.  Bootstrap percolation was introduced on 
trees in \cite{CLR} and has been since extensively studied on various graphs, with 
many deep and surprising results, beginning with early papers \cite{vE, AL}. Particularly impressive 
are results on $\bZ^d$; for some  of the highlights, see \cite{Hol, HLR, BBDM, HMo, BDMS}, recent 
papers \cite{BBMS1, BBMS2}, and survey \cite{Mor}, which contains a wealth of further references. Analysis 
of bootstrap percolation on graphs with longer range connectivity,
related to the present setup, is more recent. It was introduced in \cite{GHPS} and 
further explored in \cite{Sli, GSS, GS1, GS2}. 

Another special case is {\it line growth\/}, with finite rectangular zero-set 
$\cZ=[0,r-1]\times [0,s-1]$, where $1\le s\le r$. These dynamics were introduced 
 as {\it line percolation\/} \cite{BBLN, GSS} on Hamming graphs (which have $\rho=\infty$, i.e., the neighborhood is an infinite cross in 
both directions). On Hamming graphs, these dynamics have the property that 
any point gets occupied together with an entire line through it, which is of great help 
in the analysis. It turns out that this property approximately holds in the local version 
of the present paper and yields our main results. On the other hand, we suspect, 
and are able to prove in one case, logarithmic corrections to the power laws when $r=s>1$, 
which have no counterpart on the Hamming plane and are somewhat surprising for 
symmetric supercritical rules. See Figure~\ref{fig-sim} for simulations 
of a bootstrap percolation and a line growth dynamics.

\begin{figure}
\begin{center}
\includegraphics[scale=0.3]{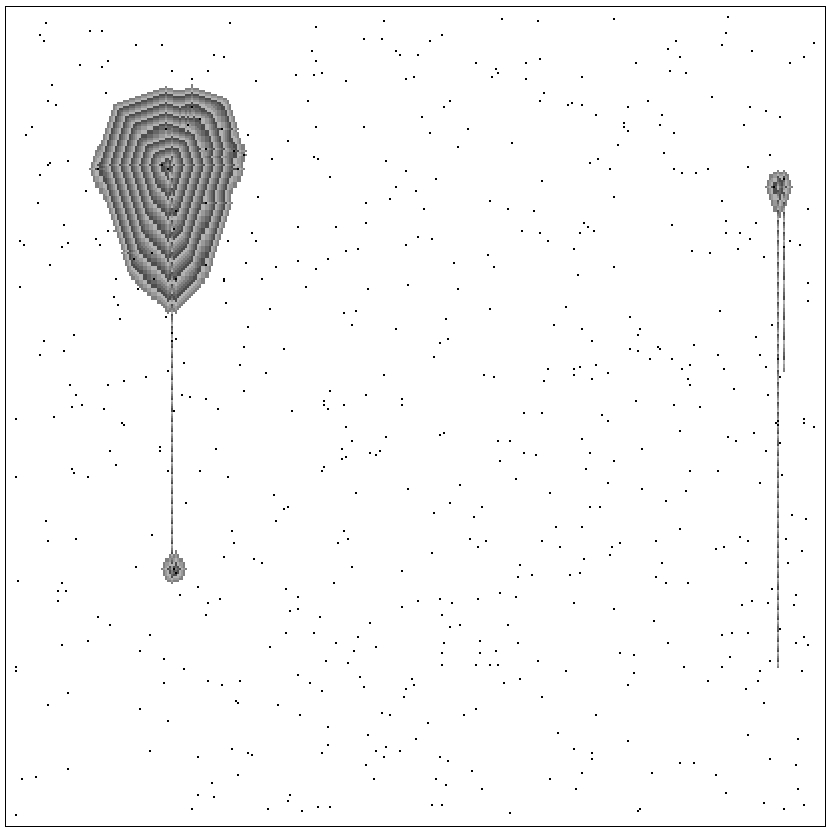}\hskip0.2cm\includegraphics[scale=0.3]{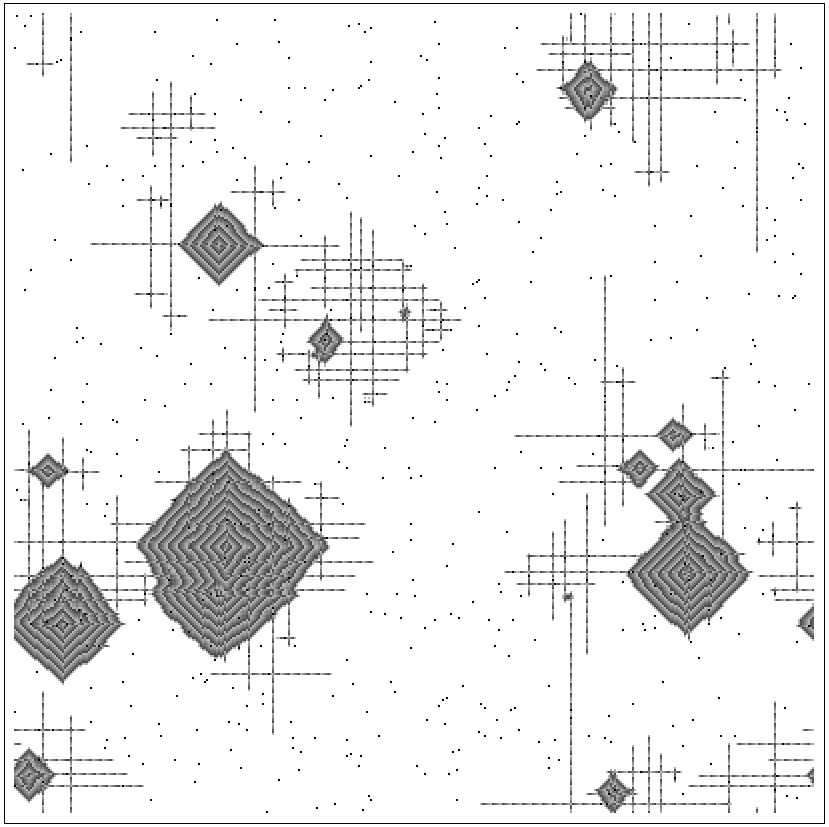}
\end{center}
\caption{Simulation of bootstrap percolation with $r=3$ with $\rho=3$ at $p=0.003$ (left) and 
symmetric line growth with $r=2$ with $\rho=2$ at $p=0.002$ (right). The still-frames are taken 
at the time when large occupied sets are about to quickly  take over the space. Sites occupied at
time $t>0$ are colored periodically with shades of grey.}
    \label{fig-sim}
\end{figure}

Assume that $a$ is a deterministic sequence depending on $p$ and $X$ is a
sequence of nonnegative random variables depending on $p$. We write:
\begin{itemize}
\item
 $X\lesssim a$ if 
$\lim_{\lambda\to \infty}\limsup_{p\to 0} \prob{X\ge \lambda a}= 0$; 
\item $X\gtrsim a$ if 
 $\lim_{\lambda\to 0}\limsup_{p\to 0} \prob{X\le \lambda a}=0$; and 
 \item $X\asymp a$ if 
 $X\lesssim a$ and $X\gtrsim a$. 
 \end{itemize} 
 In this sense, $X\asymp a$ means that $X=\Theta(a)$ 
 in probability.
 
 We call $\gamma>0$ a {\it lower power\/} (resp., an {\it upper power\/}) for $\cZ$
if, for 
every $\epsilon>0$, $T\gtrsim p^{-(\gamma-\epsilon)}$ (resp., $T\lesssim p^{-( \gamma+\epsilon)}$).
 Then $\gc=\gc(\cZ)$ is the {\it critical power\/} for the zero-set $\cZ$
if it is both an upper and a lower power. 
It follows from~\cite{BSU} (see also \cite{BBMS1}) that finite lower and upper powers always exist 
under our assumption on $\cZ$ and $\rho$ that guarantee supercriticality. 
However, there seems to be no general method that would prove that 
$\gc$ always exists, and even good inequalities may be very 
difficult if not impossible to obtain in general; see \cite{BBMS2, HMe} for discussion on 
related computational problems. 
Our methods demonstrate the existence of $\gc$ only
when it is possible to compute it exactly, which includes many small zero-sets; see Section~\ref{sec-examples}.


Existence of a critical power still allows for, say, logarithmic corrections in the scaling of $T$, so it 
is meaningful to ask whether such corrections are absent. 
We call the critical power $\gc$ {\it pure\/} if $T\asymp p^{-\gc}$. We now proceed to 
stating the main results.

%
 
 \begin{theorem}\label{thm-bp}
  Assume bootstrap percolation $\cZ$ with threshold $r\ge 1$. 
Let
$$
\hm=\left\lceil\frac{\sqrt{9+8r}-5}{2}\right\rceil.
$$
Then the pure critical power is
\begin{align*}
\gamma_c=\frac{(\hm+1)(2r-\hm)}{2(\hm+2)}.
\end{align*}
\end{theorem}

We note that the upper bounds implied by Theorem~\ref{thm-bp} play a role
in the determination of the critical length for three dimensional critical bootstrap percolation processes (see Proposition 1.2 in \cite{Bla2}), so our main contribution are the matching lower bounds.

 \begin{theorem} \label{thm-line}
 Assume line growth $\cZ$ with parameters $2\le s \le r$. 
Then the critical power is
\begin{align*}
\gamma_c=\frac{(r-1)s}{r}.
\end{align*}
\end{theorem}

Line growth with $1=s\le r$ is different: the critical power is $r/(r+1)$ and is pure 
(see Proposition~\ref{prop-s=1}).
In general, we are not able to determine purity for line growth with $s\ge 2$. We can, however, 
demonstrate that the critical power is not pure in one case. Also, to our knowledge, this is the only symmetric growth rule for which the scalings of $T$ and $L_c$ provably differ.

\begin{theorem} \label{thm-liner2}
Assume line growth $\cZ$ with parameters $r=s=2$. 
Then 
\begin{align*}
T\asymp p^{-1}\log p^{-1},
\end{align*}
and $L_c \asymp p^{-1}$.
\end{theorem}

Next, we consider L-shaped Young diagrams, for which we can only determine $\gamma_c$ in 
special cases, and we instead give power bounds in general.


\begin{theorem}\label{thm-L-finite}
Assume $\cZ$ is given by the minimal counts $(0,r)$, $(s_1,s_2)$, and $(r,0)$, and assume
$1\le s_1,s_2<r$. 
 
\begin{enumerate}
\item If $s_1=1$ and $s_2=s$ with $s\le r/2$, the pure critical power is $\gamma_c=r/2$.
\item If $s_1=s_2=2$ and $r\ge 6$, the pure critical power is $\gamma_c=2r/3$.
\item If $s_1=s_2=s$, a lower and an upper power are, respectively,
$$\gamma_\ell=
 \frac{\hm(r-\hm+1)}{1+\hm}, \text{ where }\hm=\min(\lfloor(-1+\sqrt{4r+9})/2\rfloor,\lfloor s/2\rfloor),
$$
and
$$
\gamma_u= \frac{rs}{s+1}.
$$ 

\item Assume $s_1=1$ and $s_2=s$. 
For $s>r/2$, 
we have lower and upper power $s-\frac{s}{r}$ and  $s+1-\frac{2s+1}{r+1}$, which differ 
by at most $1/2$. 
\end{enumerate}
\end{theorem}

Assume the zero set $\cZ$ has height and width $r$, and $r$ is large. If $\cZ$ is between 
bootstrap percolation and symmetric line growth with the same threshold $r$, Theorems~\ref{thm-bp} 
and~\ref{thm-line} give lower and upper powers of the form $r-\cO(\sqrt r)$.  On the other hand, 
concave 
zero-sets may no longer have powers of the form $r-o(r)$. For example, the powers are $r/2$ 
for the zero-set of part 1 
of Theorem~\ref{thm-L-finite}, but if both arms have thickness $s$ they 
change to $2r/3$ when $s=2$ (part 2) and 
become $r-o(r)$ when $s\to\infty$ and $r-\Theta(\sqrt r)$ when $s=\Theta(\sqrt r)$ (part 3). 
By contrast, if one arm of the zero-set has thickness $1$ and the other has thickness $s$,
the powers do not change at all up to 
$s=\lfloor r/2\rfloor$ and become $r-o(r)$ only when $s=r-o(r)$ (parts 1 and 4). 

Finally, we address a family of infinite zero-sets, for which the critical power can again 
be established.

\begin{theorem}\label{thm-L-infinite}
Assume the infinite zero set is given by the minimal counts  $(0,r)$ and $(s_1,s_2)$, 
where $1\le s_2<r$ and $1\le s_1$. Then the pure critical power is 
$$
\gamma=\frac{rs_1+s_2}{1+s_1}. 
$$
\end{theorem}

The rest of the paper is organized as follows. In Section~\ref{sec-prelim} we give definitions and
simple preliminary results that we routinely use throughout the paper. In Section~\ref{sec-examples}
we address the 13 zero-sets that fit into $3\times 3$ box, assuming the results 
from subsequent sections. Section~\ref{sec-bp} is devoted to bootstrap percolation and 
proof of Theorem~\ref{thm-bp}. Section~\ref{sec-line-sym-upper} gives the upper bound for 
the symmetric line growth for arbitary $r\ge 2$, which 
is more precise than stated in Theorem~\ref{thm-line}, and is needed for Theorem~\ref{thm-liner2},
whose proof is completed in Section~\ref{sec-line-2}. The most substantial and technical is 
Section~\ref{sec-line-general}, which completes the proof of Theorem~\ref{thm-line}. At hart, that 
proof is an optimization argument, verifying which nucleation scenarios are optimal. 
The next two sections 
address L-shaped Young diagrams, finite ones in Section~\ref{sec-L-finite} 
(proving Theorem~\ref{thm-L-finite}) and infinite ones in Section~\ref{sec-L-infinite} 
(proving Theorem~\ref{thm-L-infinite}). 
We conclude with a selection of open problems in Section~\ref{sec-open}. 
 
\section{Preliminaries}\label{sec-prelim}

%

We start by introducing a convenient notation for tracking the powers of $p$. 
 For a sequence $X$ of nonnegative random variables, we write $X\gep \alpha$ if  $X p^{\alpha}\to \infty$
 in probability, and  $X\lep{\alpha}$ if 
 $X p^{\alpha}\to 0$
 in probability.
If $\alpha\le 0$ and $X$ is an integer-valued random variable, the meaning of $X\lep{\alpha}$ is simply that $X\to 0$ in probability. We artificially declare $X\gep{\alpha}$ whenever $\alpha<0$.

The next lemma, whose proof is a simple exercise, provides a useful reformulation 
of the asymptotic properties.

\begin{lemma} \label{lemma-prelim-1}
Let $a$ be a deterministic sequence depending on $p$ and $X$ a
sequence of nonnegative random variables depending on $p$. Then 
$X\lesssim a$ if and only if $\prob{X\le t}\to 1$ for every sequence 
$t$ such that $t\gg a$, and $X\gtrsim a$ if and only if $\prob{X\le t}\to 0$ for every sequence 
$t$ such that $t\ll a$. 
Moreover, $X\lep \alpha$ and  $X\gep \alpha$
can be respectively characterized as $\prob{X\le t}\to 1$ for any $t\gtrsim p^{-\alpha}$ and 
$\prob{X\le t}\to 0$ for any $t\lesssim p^{-\alpha}$.
\end{lemma} 

We next state two simple lemmas which we routinely use throughout. As mentioned in 
Section~\ref{sec-intro}, we commonly consider 
dynamics on boxes $B_n$ with the default $0$-boundary. However, 
it is also useful to consider {\it periodic boundary\/} on $B_n$, whereby a site
$(x_1+i, x_2+j)\in\cN_{(x_1,x_2)}$ is understood to be 
$((x_1+i)\mod n, (x_2+j)\mod n)$, giving $B_n$ the
topology of the discrete torus, the Cartesian product of two cycles with $n$ vertices. In either case, 
we call $B_n$ {\it spanned at time\/} $t$ if $\xi_t=B_n$; and {\it spanned\/} if this 
holds at $t=\infty$.  We say a sequence of  events $E$ (depending on $p$) 
occurs  \textit{with high probability} if $\prob{E} \to 1$ as $p\to0$.

\begin{lemma}  \label{lemma-ub-easy} Assume that there exists a constant $C$, and a sequence $n=n(p)$ such that
$B_n$ (with the default $0$-boundary) is spanned at time ${Cn}$ with high probability. Then, with high probability, $T\le Cn$. 
\end{lemma}

\begin{proof}
This is a simple consequence of monotonicity.
\end{proof}

\begin{lemma} \label{lemma-lb-easy}
Assume that a sequence $n=n(p)$ 
is such that, with probability converging to $1$ as $p\to 0$, 
the dynamics on $B_n$ with periodic boundary has $|\xi_\infty|/|B_n|\to 0$ in probability. 
Then, with high probability, $T\ge  n/(2\rho)$. 
\end{lemma}

\begin{proof}
Due to periodic boundary, the dynamics on $B_n$ is translation invariant,  
so
\begin{align*}
\prob{\mathbf{0}\in\xi_\infty}=\expe{\frac{|\xi_\infty|}{|B_n|}}\to 0,
\end{align*}
by dominated convergence. But if $\mathbf{0}\notin\xi_\infty$, then $T\ge n/(2\rho)$ deterministically, 
as boundary effects cannot travel faster than the ``speed of light,'' that is, 
$\ell^1$-distance $\rho$ per update. 
\end{proof}

For upper bounds, we frequently use the following two inequalities.
\begin{lemma}\label{lemma-binom}
If $p$ is small enough, $\prob{\text{\rm Binomal}(n,p)\le np/2}\le e^{-np/7}$.
If $np$ is small enough, $\prob{\text{\rm Binomal}(n,p)>0}\ge np/2$.
\end{lemma}

On the other hand, we use the following consequence of the Markov inequality for 
lower bounds. 

\begin{lemma}  \label{lemma-markov} If $X$ is a sequence of random variables that depends on $p$, 
and $a$ is a deterministic such sequence, then $\mathbb E X\ll a$ implies $X/a\to 0$ in 
probability. In particular, for $\gamma\in\bR$, $\mathbb E X\lep \gamma$ implies 
$X\lep \gamma$.
\end{lemma}
 
We often use color-coding to distinguished between sites that get occupied in different 
times in our various updating schemes. In all schemes, we consistently refer to initially 
occupied sites as {\it black\/}, and nonoccupied vertices as {\it white\/}.

\section{Examples: all zero-sets with height and width at most 3}\label{sec-examples}

Up to reflection symmetry, there are 13 zero-sets that fit into the $3\times 3$ box. 
Our methods provide powers in all 13 cases, and the results are below. 
Three cases are not completely covered by our results elsewhere in the paper, so
we provide separate arguments, which also serve as simple illustrations of the methods we use. 


\ytableausetup{boxsize=0.3cm, centertableaux , 
}
\begin{itemize}
\item $\cZ=\ydiagram[*(ashgrey)]{1}$\hskip0.2cm This is the only voracious case, which is easily seen 
to have pure critical power $\gc=1/2$. 

\item $\cZ=\ydiagram[*(ashgrey)]{2}$\hskip0.2cm Pure critical power $\gc=2/3$ 
(Proposition~\ref{prop-s=1}).

\item $\cZ=\ydiagram[*(ashgrey)]{3}$\hskip0.2cm Pure critical power $\gc=3/4$ (Proposition~\ref{prop-s=1}).

\end{itemize}

\begin{prop}\label{prop-s=1}
The line growth with $1=s\le r$ has pure critical power $\gc=r/(r+1)$. 
\end{prop}
\begin{proof}
If $n\gg p^{-r/(r+1)}$, that is, 
$n(np)^r\gg 1$, then with high probability
the box $B_n$ contains $r$ initially nonempty neighboring vertical lines. Such a configuration spans $B_n$ by time $2n$, so Lemmas~\ref{lemma-prelim-1} and~\ref{lemma-ub-easy} finish the proof of 
the upper bound. To prove the matching lower bound, consider the dynamics on $B_n$ with 
periodic boundary. If no $(2\rho+1)\times n$ vertical strip contain $r$ initially nonempty vertical lines, no site in 
any initially empty vertical line gets occupied. This happens with high probability if
$n\ll p^{-r/(r+1)}$, in which case also the expected number of sites on initially nonempty vertical lines
is at most a constant times $n^2\cdot np\ll n^2$. Lemma~\ref{lemma-lb-easy} then implies the lower 
bound. 
\end{proof}

\begin{itemize}

\item $\cZ=\ydiagram[*(ashgrey)]{1, 2}$\hskip0.2cm This is bootstrap percolation with $r=2$ and pure critical power $\gc=1$ (Theorem~\ref{thm-bp}).

\item $\cZ=\ydiagram[*(ashgrey)]{1, 3}$\hskip0.2cm Pure critical power $\gc=1$ 
(Proposition~\ref{prop-small-L}). 

\end{itemize}

\begin{prop}\label{prop-small-L}
The zero-set with minimal counts $(0,2)$, $(1,1)$ and $(r,0)$, with $r\ge 2$, has 
pure critical power $\gc=1$.
\end{prop}

\begin{proof}
The lower bound 
follows by comparison with bootstrap percolation with $r=2$, whose proof can also be adapted to 
get the upper bound. Indeed, if $n\gg p^{-1}$, the probability that a fixed $r\times n$ strip has 
a vertical line with two neighboring black sites while the other $r-1$ vertical lines each have 
at least one black site is $\gg p$, and so the number of such strips in an $n\times n$ box is 
large with high probability, but one such strip will occupy the entire box by time $4n$, 
and the upper bound follows by Lemmas~\ref{lemma-prelim-1} and~\ref{lemma-ub-easy}.
\end{proof}

\begin{itemize}

\item $\cZ=\ydiagram[*(ashgrey)]{2, 2}$\hskip0.2cm This case still has critical power $\gc=1$, which 
however is 
no longer pure. From Theorem~\ref{thm-liner2}, we get that $T\asymp p^{-1}\log p^{-1}$.  

\item $\cZ=\ydiagram[*(ashgrey)]{2, 3}$\hskip0.2cm Critical power $\gc=4/3$, and 
purity is unresolved (Proposition~\ref{prop-line-perturbed}). 

\item $\cZ=\ydiagram[*(ashgrey)]{3, 3}$\hskip0.2cm This is line percolation with $r=3$ and $s=2$, so it has the same critical power $\gc=4/3$, with 
unresolved purity (Theorem~\ref{thm-line}).  

\item $\cZ=\ydiagram[*(ashgrey)]{1,1, 3}$\hskip0.2cm Pure critical power $\gc=3/2$ 
(Theorem~\ref{thm-L-finite}).

\item $\cZ=\ydiagram[*(ashgrey)]{1,2, 3}$\hskip0.2cm This is bootstrap percolation with $r=3$ 
and pure critical power $\gc=5/3$ (Theorem~\ref{thm-bp}). 

\item $\cZ=\ydiagram[*(ashgrey)]{2,2, 3}$\hskip0.2cm Critical power still $\gc=5/3$
and still pure (Proposition~\ref{prop-223}). 

\end{itemize}

\begin{prop} \label{prop-223}
For the above zero-set, $\gc=5/3$ is the pure critical power. 
\end{prop} 

\begin{proof} 
The lower bound follows from bootstrap percolation 
with $r=3$, so we only need to prove the upper bound. 

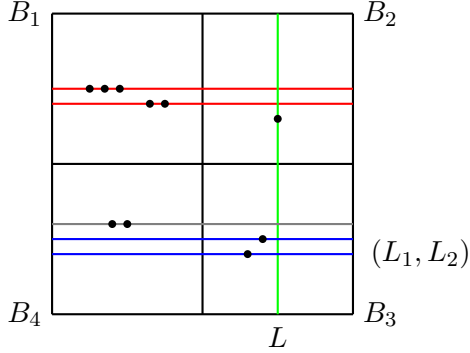
\begin{figure}
\begin{center}
\begin{tikzpicture}
\coordinate [label= left:$B_1$](b1) at (0,4);
\coordinate [label= right:$B_2$](b2) at (4,4);
\coordinate [label= right:$B_3$](b3) at (4,0);
\coordinate [label= left:$B_4$](b4) at (0,0);
\coordinate [label= center:$L$](L) at (3,-0.3);
\coordinate (L1L2) at (4,0.9);
\node at (4.9,0.8) {$(L_1,L_2)$};
\coordinate (b5) at (0,2);
\coordinate (b6) at (4,2);
\coordinate (b7) at (2,0);
\coordinate (b8) at (2,4);
\coordinate (l1) at (0,2.8);
\coordinate (r1) at (4,2.8);
\coordinate (l2) at (0,3.0);
\coordinate (r2) at (4,3.0);

\coordinate (d1) at (3.0,0);
\coordinate (u1) at (3.0,4);

\coordinate (l3) at (0,0.8);
\coordinate (r3) at (4,0.8);
\coordinate (l4) at (0,1.0);
\coordinate (r4) at (4,1.0);
\coordinate (l5) at (0,1.2);
\coordinate (r5) at (4,1.2);
 
\coordinate (p1) at (1.3,2.8);
\coordinate (p2) at (1.5,2.8);
\coordinate (p3) at (0.5,3);
\coordinate (p4) at (0.7,3);
\coordinate (p5) at (0.9,3);
\coordinate (p6) at (3,2.6);
\coordinate (p7) at (2.8,1);
\coordinate (p8) at (2.6,0.8);
\coordinate (p9) at (1,1.2);
\coordinate (p10) at (0.8,1.2);

 \draw[color=black, thick] (b1)--(b2);
 \draw[color=black,thick] (b2)--(b3);
 \draw[color=black,thick] (b3)--(b4);
\draw[color=black,thick] (b4)--(b1);
\draw[color=black,thick] (b5)--(b6);
\draw[color=black,thick] (b7)--(b8);
\draw[color=red, thick] (l1)--(r1);
\draw[color=red, thick] (l2)--(r2);
\draw[color=green, thick] (u1)--(d1);
\draw[color=blue, thick] (l3)--(r3);
\draw[color=blue, thick] (l4)--(r4);
\draw[color=gray, thick] (l5)--(r5);

\foreach \point in {p1,p2,p3,p4,p5,p6, p7,p8, p9,p10}
           \fill [black] (\point) circle (1.5pt);
\end{tikzpicture}
\end{center}
\caption{Occupying 3 neighboring parallel horizontal lines for the zero set
$\cZ=\ydiagram[*(ashgrey)]{2,2, 3}$. The red lines get occupied first, then the green line $L$, then 
the two blue lines $(L_1,L_2)$, and finally the grey line.}
    \label{fig-331}
 \end{figure}

Assume $n\gg p^{-5/3}$. Divide the box $B_{2n}$ into 4 congruent $n\times n$ boxes $B_{1}$ (top left), $B_{2}$ (top right), $B_{3}$ (bottom right), and $B_4$ (bottom left); see Figure~\ref{fig-331}. 
The probability of the event that 
there is a horizontal line with 3 contiguous black sites next to a horizontal line with 2 neighboring black sites, both in $B_1$, approaches $1$. Conditioned on this event, there is at least 
one pair of fully occupied neighboring horizontal lines in $B_2$ (by time $4n$). 
Let $V_1$ be the number of vertical 
lines that contain at least one black site in the neighborhood of the pair. As $np\gg p^{-2/3}$, $V_1\gep 2/3$. We may also assume that these lines are at distance at least $2$ from one another. All $V_1$ lines become occupied (by additional time $2n$). Let $H_2$ be the number of neighboring pairs
of horizontal lines $(L_1, L_2)$, such that there exists a line $L$ among the $V_1$ eventually occupied 
vertical lines, and two non-collinear black sites in $B_3$ and in the neighborhood of $L$, one on $L_1$ and one on $L_2$. Then $L_1$ and $L_2$ become occupied (by additional time $2n$).
We also assume that these pairs are separated by distance at least $1$. 
As $p^{-2/3}np^2\gg p^{-1/3}$, $H_2\gep 1/3$. Now consider the event that there are two neighboring
black sites on a horizontal line in $B_4$ and in the neighborhood of one of the $H_2$ 
horizontal pairs. As $p^{-1/3}np^2\gg 1$, this event happens with high probability, and results in 3 occupied neighboring horizontal lines 
by additional time $2n$. Then, by additional time $2n$, the entire $B_{2n}$ is finally occupied. Thus, with 
high probability, $B_{2n}$ is occupied by time $12n$, and Lemmas~\ref{lemma-prelim-1} and~\ref{lemma-ub-easy} finish
the argument.
\end{proof}

\begin{itemize}

\item $\cZ=\ydiagram[*(ashgrey)]{2,3, 3}$\hskip0.2cm Critical power $\gc=2$ 
with unresolved purity (Proposition~\ref{prop-line-perturbed}). 

\item $\cZ=\ydiagram[*(ashgrey)]{3,3, 3}$\hskip0.2cm This is symmetric line growth with $r=3$ 
and critical power $\gc=2$ (Theorem~\ref{thm-line}), which we suspect, 
due to Theorem~\ref{thm-liner2}, 
not to be pure. 

\end{itemize}

Two small examples for which our methods leave critical powers
unresolved 
are $\cZ=\ydiagram[*(ashgrey)]{2,2, 4}$ (with lower and upper powers $5/3$ and $2$), 
and $\cZ=\ydiagram[*(ashgrey)]{3,3, 4}$ (with lower and upper powers $2$ and 
$9/4$).



%


\section{Bootstrap percolation}\label{sec-bp}

\subsection{Upper bounds}\label{sec-bp-upper}

For an integer $k\ge 1$, call a row of the box $B_n$ $k$-filled if there is a 
contiguous interval of length $k$ consisting entirely of black sites.  
Assume $n\ge k$. If $np^k<1$, the probability that a fixed row is $k$-filled 
is for small $p$ between $\frac{1}{4k}\cdot np^k$ and $np^k$. 
If $np^k\ge 1$, this probability is at least $\frac{1}{4k}$, and it goes to $1$ 
if $np^k\gg 1$.

A strip of $r$ consecutive rows in $B_n$ is {\it packed\/} if it has the following property: 
starting from the bottom of the strip, the first row is $r$-filled, the next row
is $(r-1)$-filled, \ldots, and the top row is $1$-filled. If a strip is packed, then it is easy to see that it is spanned in at most $rn$ time steps and then the entire $B_n$ is spanned in 
at most additional $n$ time steps.

Assume $r\ge 2$, $p$ is small, and $n\sim Cp^{-\gamma}$ for some large constant $C$. The probability that a fixed strip is packed is at 
least a constant times
$$
np^r\cdot np^{r-1}\cdots np^{r-\bm},
$$
where $\bm=\bm(\gamma)\in[1,r]$ is the largest integer $m$ such that $np^{r-m}< 1$, that is, 
such that $-\gamma+r-m> 0$. There 
are $\lfloor n/r\rfloor$ independent candidates for a packed strip in $B_n$. Therefore, 
with this choice of $\bm$, the probability that there exists a packed strip is close 
to $1$ as soon as
\begin{equation}\label{packed}
n\cdot np^r\cdot np^{r-1}\cdots np^{r-\bm}=n^{\bm+2} p^{(\bm+1)(2r-\bm)/2}
\end{equation}
is large. 
Let $\hm$ be the largest $m$ for 
which
$$
(r-m)-\frac{(m+1)(2r-m)}{2(m+2)}> 0,
$$
that is, 
$$
-m^2-3m+2r> 0.
$$
Then we pick $\gamma$ so that 
the power of $p$ in (\ref{packed}) vanishes if $\bm$ is replaced by $\hm$
(so that (\ref{packed}) is a large constant).
These computations imply the following result. 
\begin{lemma}\label{lemma-bp-upper}
For $r\ge 2$ let
$$
\hm=\left\lceil\frac{\sqrt{9+8r}-5}{2}\right\rceil
$$
and 
\begin{equation}\label{gamma-choice}
\gamma=\frac{(\hm+1)(2r-\hm)}{2(\hm+2)}.
\end{equation}
Then $n\gg p^{-\gamma}$ implies $\prob{B_n\text{ spanned at time $(r+1)n$}}\to 1$. Consequently, 
$t\gg p^{-\gamma}$ implies $T\lesssim t$.
\end{lemma}

\begin{proof} 
We need to check that, for the choice of $\gamma$ in (\ref{gamma-choice}), $\bm(\gamma)=\hm$; this is equivalent 
to $-(\hm+1)^2-3(\hm+1)+2r\le 0$, which holds by definition of $\hm$. 
The last statement follows from 
Lemmas~\ref{lemma-prelim-1} and~\ref{lemma-ub-easy}.
\end{proof}
We will prove the matching lower bounds in the next subsection, so in 
fact $\gamma$ is the critical power. Its values, together with the corresponding $\hm$'s,
for small $r$ are in Table~\ref{table-bp}. 
For large $r$, $\hm=\sqrt{2r}+\mathcal O(1)$, and consequently
$$
\gamma= r-\sqrt{2r}+\mathcal O(1).
$$

\begin{table}
    \centering
    \begin{tabular}{|c||c|c|c|c|c|c|c|c|c|c|c|c|c|c|c|c| c|c|c|c|c|}
    \hline 
    $r$ & 1 & 2 & 3 & 4 & 5 & 6 & 7 & 8 & 9 & 10 & 11 & 12 & 13 & 14 & 15 & 16& 17& 18& 19& 20\\ [3mm]
    \hline
   $\gamma_c$ & $\frac{1}{2}$ & 1 & $\frac{5}{3}$ & $\frac{7}{3}$ & 3 & $\frac{15}{4}$ & $\frac{9}{2}$ & $\frac{21}{4}$ & 6 & $\frac{34}{5}$ & $\frac{38}{5}$ & $\frac{42}{5}$ & $\frac{46}{5}$ & 10 & $\frac{65}{6}$ & $\frac{35}{3}$& $\frac{25}{2}$& $\frac{40}{3}$& $\frac{85}{6}$ & $15$\\ [2mm]
   \hline
   $\hm$ & $0$ & 0 & $1$ & $1$ & $1$& $2$ & 
  $2$ & $2$ & $2$ & $3$ & $3$ & $3$ & $3$ & $3$ & $4$ & $4$& $4$& $4$& $4$& $4$\\ [2mm]
    \hline
\end{tabular}
 \caption{The critical powers $\gamma_c$ and the corresponding $\hm$ for small thresholds.}
    \label{table-bp}
    \end{table}

\medskip

\subsection{Lower bounds}\label{sec-bp-lower}

\begin{lemma}\label{lemma-bp-lower} 
For $r\ge 2$, let $\hm$ and $\gamma$ be as in Lemma~\ref{lemma-bp-upper}.
Then $p^{-(\gamma-1)}\ll n\ll p^{-\gamma}$  implies $\prob{B_n\text{ spanned}}\to 0$. Moreover, 
if $t\ll p^{-\gamma}$, then $\prob{T\le t}\to 0$. 
\end{lemma}
 
 
We use the following growth dynamics as an upper bound for our process on $B_n$ with periodic boundary. As will be the case throughout the paper, we find a color-coding useful. Recall 
that initially occupied and non-occupied points are black and white. We will also 
use red and blue colors, none of which are present initially. Fix an ordering of the vertices
of $B_n$. At step $\ell$, we examine the white vertices in order, and if a vertex has at least a combined number $r$ of black and blue neighbors, then it is marked red, 
along with all points in the same row and column of $B_n$. At
the end of a step, all red points switch to blue. Note that if a vertex is marked red before it is 
examined, then it will be skipped.


For an integer $m\ge 1$, we say that a collection of $m$ rows (resp. columns) form an 
{\it $m$-clump\/} if they are within distance $3\rho$ 
of each other and they are fully blue. 
(The clumps are not necessarily
disjoint, so for example every $m$-clump contains $\binom m {k}$ $k$-clumps for $k<m$.)
Let $L_m^{(\ell)}$ denote the number of $m$-clumps of lines (rows or columns) after round $\ell$. We let $L_0^{(\ell)}=n$. 
Let $\Delta L_m^{(\ell)} = L_m^{(\ell)} - L_m^{(\ell-1)}$ denote 
the number of new $m$-clumps added at step $\ell$.

Let $\gamma$ be given by (\ref{gamma-choice}). For $m\ge 0$, let 
$$\beta(m)=\gamma(m+1) - \frac{m(2r-m+1)}{2}=\gamma(m+1) - rm+\frac{m(m-1)}{2}.
$$ 
Recall that $\hm$ is the largest integer $m$ for which $-\gamma+r-m>0$. Also, 
$\beta(\hm+1)=0$. 


\begin{lemma} \label{lemma-bp-lb-dec}
For $0\le m\le \hm$, $\beta(m)+m$ is a (strictly) decreasing function of $m$. Consequently, 
$\beta(m)+m\le\gamma$. 
\end{lemma}
\begin{proof} 
For $m<\hm$,
\begin{align*}
\beta(m)-\beta(m+1)=-\gamma+r-m =-\gamma+r-(m+1)+1>1, 
\end{align*}
which, together with $\beta(0)=\gamma$, implies the two statements.
\end{proof}

\begin{lemma} \label{lemma-bp-lb-key} Assume $n\ll p^{-\gamma}$. 
For each $1\le \ell\le \hm+1$ and $0\le m\le \hm+1$,
 \begin{align*}
 L_m^{(\ell)} \lep \beta(m).
 \end{align*}
Moreover, for each $1\le \ell\le \hm+1$ and $1\le m\le \ell$,
 \begin{align*}
\Delta L_m^{(\ell)} \lep \beta(\ell).
 \end{align*}
 \end{lemma}
 \begin{proof} 
Let $\cL$ be a fixed (deterministic) set of red and blue (horizontal and vertical) lines 
in $B_n$. Condition on the event $F$ that
exactly they are colored at a certain point by the described comparison occupation dynamics.
By FKG, under this conditioning, 
the probability of any increasing event $E$ depending on the configuration 
outside $\cL$ is smaller than the non-conditional probability of $E$.
This is because the event $F$ is decreasing as a function 
of the configuration on $B_n\setminus \cL$.
 
We proceed by induction on $\ell$. For $\ell=1$, the expected number of $m$-clumps 
is at most a constant times $n\cdot (np^r)^m$ and so 
\begin{align*}
 \Delta L_m^{(1)}=L_m^{(1)} \lep (m+1)\gamma-rm\le \beta(m).
 \end{align*}


We now proceed to proving the $\ell\to\ell+1$ induction step, for $\ell\le \hm$. 
An $m$-clump is created by adding $m-a$ parallel lines to some $a$-clump, for some 
$a\in [0, m-1]$. Each of those lines may receive help from a perpendicular $b_i$-clump, where 
$0\le b_i\le r-a$, $i=1,\ldots m-a$. 

We now have two possibilities. The first is that the $a$-clump was created 
at step $\ell$. By the induction hypothesis, for
$a\le \ell$, $\Delta L_a^{(\ell)}\le \beta(\ell)$. As the $\ell$-clumps 
offer the most help for expansion, we may assume 
that 
$a\ge \ell$ and so $m\ge\ell+1$. For fixed $a$ and $b_i$, we use our observation 
on negative correlations due to the FKG inequality and conditioning
on all $L_k^{(\ell')}$ and $\Delta L_k^{(\ell')}$ for $\ell'\le \ell$, to get that,
with probability tending to $1$ as $C\to\infty$,
\begin{align*}
\Delta L_{m}^{(\ell+1)}\le C \Delta L_a^{(\ell)} \prod_i L_{b_i}^{(\ell)}p^{r-a-b_i}.
\end{align*}
It follows from the induction hypothesis that
\begin{align*}
\Delta L_m^{(\ell+1)}&\lep \max_{a, b_i}\left[\beta(a)+\sum_i(a-r+b_i+\beta(b_i))\right]\\
&=\max_{a, b_i}\left[\beta(a)+(m-a)(a-r)+\sum_i(b_i+\beta(b_i))\right]\\
&= \max_{a}\left[\beta(a)+(m-a)(a-r+\gamma)\right]\\
&= \max_{a}\left[\beta(m)-\frac12(m-a)(m-a-1)\right]\\
&=\beta(m).
\end{align*}
Here, we used Lemma~\ref{lemma-bp-lb-dec} on the third line, and the observation that 
the maximum on the penultimate line is achieved at $a=m-1$.

The second possibility is that {\it every\/} $b_i$-clump was created at step $\ell$. In this 
case, the inequality that holds with probability tending to $1$ as $C\to\infty$ is
\begin{align*}
\Delta L_{m}^{(\ell+1)}\le C L_a^{(\ell)} \prod_i \Delta L_{b_i}^{(\ell)}p^{r-a-b_i}.
\end{align*}
We may assume that $a\le \hm$, as it follows from the induction hypothesis that 
there is no $(\hm+1)$-clump at step $\ell$. When $m\ge \ell$,
\begin{align*}
\Delta L_m^{(\ell+1)}&\lep \max_{a, b_i}\left[\beta(a)+\sum_i(a-r+b_i+\beta(\max(b_i,\ell)))\right]\\
&=\max_{a, b_i\ge \ell}\left[\beta(a)+\sum_i(a-r+b_i+\beta(b_i))\right]\\
&=\max_{a, b_i\ge \ell}\left[\beta(a)+(m-a)(a-r)+\sum_i(b_i+\beta(b_i))\right]\\
&= \max_{a}\left[\beta(a)+(m-a)(a-r+\ell+\beta(\ell))\right]\\
&\le \max_{a}\left[\beta(a)+(m-a)(a-r+\gamma)\right]\\
&=\beta(m), 
\end{align*}
as before. One the other hand, when $m\le \ell$, we join the above computation 
on the fourth line, and use $\beta(a)\le \gamma-a$,  to get
\begin{align*}
\Delta L_m^{(\ell+1)}&\lep \max_{a}\left[\gamma-a+(m-a)(a-r+\ell+\beta(\ell))\right].
\end{align*}
Differentiating the expression inside the maximum with respect to $a$, we obtain
\begin{align*}
m+r-2a-\beta(\ell)-\ell-1&\ge m+r-2(m-1)-\gamma-1\\
&=r+1-m-\gamma\\
&\ge r-\hm-\gamma>0,
\end{align*}
which implies that the maximum is obtained at $a=m-1$, and then
\begin{align*}
\Delta L_m^{(\ell+1)}&\lep \gamma-m+1+m-1-r+\ell+\beta(\ell)\\
&=\gamma-r+\ell+\beta(\ell)\\
&= \beta(\ell+1).
\end{align*}
This completes verification of the induction step and ends the proof.
 \end{proof}
 
 \begin{proof}[Proof of Lemma~\ref{lemma-bp-lower}] By Lemma~\ref{lemma-bp-lb-key}, 
$\Delta L_1^{(\hm+1)}=0$ with probability tending to $1$ as $p\to 0$, which implies 
that our comparison dynamics stops with high probability at step $\hm$, 
which is independent of $p$. By 
Lemma~\ref{lemma-bp-lb-key} and Lemma~\ref{lemma-bp-lb-dec}, 
$L_1^{(\hm)}\lep\gamma-1$ and so $L_1^{(\hm)}\ll p^{-(\gamma-1)}\ll n$ with high probability. 
Therefore, after the step $\hm$, the number of blue points is $o(n^2)$ with high probability. 
The number of black points has expectation $pn^2$ and is therefore also 
$o(n^2)$ with high probability. This proves the first statement and 
Lemma~\ref{lemma-lb-easy} proves the second one. 
 \end{proof}
 
 \begin{proof}[Proof of Theorem~\ref{thm-bp}]
 The conclusion follows immediately from 
Lemmas~\ref{lemma-bp-upper} and~\ref{lemma-bp-lower}.
 \end{proof}

\section{Symmetric line growth: general upper bound}\label{sec-line-sym-upper}


\begin{lemma}\label{lemma-line-sym-upper} Assume $r\ge 2$. If $t\gg {p^{-(r-1)}}\log\frac1p$, then $T\lesssim t$. Moreover,  if $n\gg p^{-(r-1)}$, then $B_n$ is spanned with high probability.
\end{lemma} 

\begin{proof} 
Let $n=A/p^{r-1}$.
Divide the $n(r-1)\times n(r-1)$ box $S$ with $\mathbf 0$ at its lower left corner
into $(r-1)\times (r-1)$ subboxes, which we call {\it lots\/}. 
A lot is {\it bottom-filled\/} (resp., {\it left-filled\/}) 
if its bottommost row (resp., leftmost column) is all black. 
A bottom-filled (resp., left-filled) lot is a {\it nucleus\/} if the bottommost row 
(resp., leftmost column) is part 
of a horizontal (resp., vertical) interval of $r$ black sites. 

In the remainder of the proof, we consider the density of black sites to be $2p$ 
instead of $p$, which does not affect our claim. The black sites thus 
stochastically dominate the union of two independent configurations, 
each with density $p$. We call the two the {\it basic\/} and the {\it sprinkled\/} 
configuration. Until further notice, all events will be with respect to the 
basic configuration.

Let $n=A/p^{r-1}$.
Divide the $n(r-1)\times n(r-1)$ box $S$ with $\mathbf 0$ at its lower left corner
into $(r-1)\times (r-1)$ subboxes, which we call {\it lots\/}. 
A lot is {\it bottom-filled\/} (resp., {\it left-filled\/}) 
if its bottommost row (resp., leftmost column) is all black. 
A bottom-filled (resp., left-filled) lot is a {\it nucleus\/} if the bottommost row 
(resp., leftmost column) is part 
of a horizontal (resp., vertical) interval of $r$ black sites.

Within this proof, the columns 
of $S$ are columns of lots, and analogously for rows. Thus $S$ has $n$ rows and 
columns. 

We will create a set of marked lots by the following sequential procedure. At each 
step, in addition to newly marking lots, we enlarge a set of rows and columns that are removed 
from future consideration and are called {\it stamped\/}. Initially, no lots are marked 
and we stamp the rightmost column and the topmost row.  

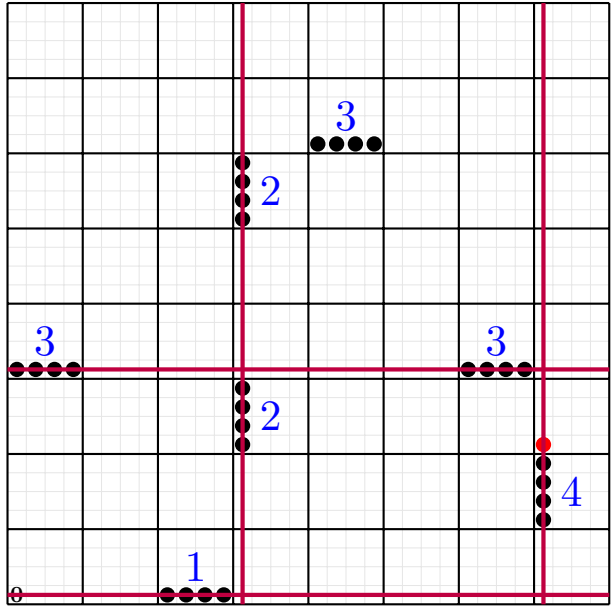
\begin{figure}
\begin{center}
\begin{tikzpicture}
\draw[step=0.25cm,faint,very thin] (0,0) grid (8,8);
\draw[step=1cm,black,thick] (0,0) grid (8,8);


\node at (0.125,0.125) {\scalebox{0.8}{{\color{black}$\mathbf{0}$}}};

 \foreach \y in {0,1,2,3}
  \fill [black] (2.125+0.25*\y,0.125) circle (2.9pt);
  
  \node at (2.5,0.5) {\scalebox{1.5}{{\color{blue}$1$}}};
  
 \foreach \y in {0,1,2,3}
  \fill [black] (0.125+0.25*\y,3.125) circle (2.9pt);
  
  \node at (0.5,3.5) {\scalebox{1.5}{{\color{blue}$3$}}};
  
  \foreach \y in {0,1,2,3}
  \fill [black] (6.125+0.25*\y,3.125) circle (2.9pt);
  
  \node at (6.5,3.5) {\scalebox{1.5}{{\color{blue}$3$}}};
  
  \foreach \y in {0,1,2,3}
  \fill [black] (4.125+0.25*\y,6.125) circle (2.9pt);
  
  \node at (4.5,6.5) {\scalebox{1.5}{{\color{blue}$3$}}};
  
  \foreach \y in {0,1,2,3}
  \fill [black] (3.125, 2.125+0.25*\y) circle (2.9pt);
  
  \node at (3.5,2.5) {\scalebox{1.5}{{\color{blue}$2$}}};
 
  \foreach \y in {0,1,2,3}
  \fill [black] (3.125, 5.125+0.25*\y) circle (2.9pt);
  
  \node at (3.5,5.5) {\scalebox{1.5}{{\color{blue}$2$}}};
  
   \foreach \y in {0,1,2,3}
  \fill [black] (7.125, 1.125+0.25*\y) circle (2.9pt);
  
  \node at (7.5,1.5) {\scalebox{1.5}{{\color{blue}$4$}}};
  
  \fill [red] (7.125, 2.125) circle (2.9pt);
  
  \draw [color=purple,ultra thick] (7.125,0)--(7.125,8);
  \draw [color=purple,ultra thick] (0, 3.125)--(8, 3.125);
  \draw [color=purple,ultra thick] (3.125,0)--(3.125,8);
   \draw [color=purple,ultra thick] (0, 0.125)--(8, 0.125);

\end{tikzpicture}
\end{center}
\caption{An example of succesful occupation of $\mathbf 0$ 
in our branching construction for $r=5$ and $n=8$. The 
black circles are occupied in the basic configuration, as detected by the 
branching construction; the blue numbers indicate the generation at which a lot 
is marked.  The one red circle is occupied in the sprinkled configuration. The purple lines are 
successively occupied in the order $4-3-2-1$.}
\label{fig-line-supercr}
\end{figure} 

In the first step, we inspect the bottom row. If, excluding the lot  
in the stamped column, it contains no bottom-filled
lots, the process stops and the set of marked lots is empty. Otherwise, we 
mark all bottom-filled lots in this row. 
We then stamp the bottom row. 
If the process has not stopped, we proceed to the second step, 
when we consider exactly the columns immediately to the right 
of the newly marked lots. Inspect the columns from left to right. 
In any column under inspection, mark the
left-filled lots that are not in any stamped row and also not in 
the same row with any already marked lots. If no such left-filled lots exist,  
leave that column without marked lots. In any case, 
stamp all the inspected columns. 

In the third step, consider rows which are just above the 
lots marked on the second step, from bottom to top. For each of these rows, do the 
analogous procedure as on the previous step: mark all bottom-filled lots within it 
that are not in any stamped columns and not in the same column as any previously marked lot, 
and stamp it. 
  
Continue with this procedure, alternating rows and columns, until one of two 
contingencies. If in some 
step no new lots are marked, we stop after that step. On the other hand, if at some point the
number of marked lots exceeds $n/5$, we stop immediately, even if this happens within
a step that is not yet complete. 
Observe that, by our construction, 
the column immediately to the right of any 
newly marked lot in a row is not stamped before being inspected, 
and analogous property holds for the rows.

If, at some point in the procedure, the number of marked lots is $m$, and we 
proceed to inspect a row or a column, we check at least $n-2m$ lots and each of them 
is marked independently with probability $p^{r-1}$. If $m\le n/5$, the expected 
number of marked lots is at least $A/2$. Therefore, if $A$ is large enough, the 
marking process dominates a supercritical branching process up to stopping. It follows that, 
for any $\epsilon$, we can choose a large enough $A$ so that the procedure reaches 
$n/5$ marked lots in $\log  n$ steps with probability at least $1-\epsilon$. We now 
condition on this event.  

The key observation is that marked lots transfer the occupation to the origin, provided
at least one of them is ``triggered'' by an additional black site. To be more precise, 
if any of the marked lots is a nucleus, the origin becomes occupied by time $rn\log n$. 

The final step in our proof is a standard sprinkling argument. Under our conditioning, 
the union of the basic and sprinkled black sites generate 
a nucleus with probability at least $1-(1-p)^{n/5}>1-\epsilon$, when 
$A$ is large enough. 

To establish the property of $T$, for a large enough constant $A$ and small enough $p$, 
the probability that the origin is occupied by time $rn\log n\le 2Ar^2p^{-(r-1)}\log\frac1p$ is 
at least $1-\epsilon$.

By another sprinkling argument, using the same construction, we can make $r$ adjacent lines in $B_n$ occupied with probability at least $(1-\epsilon)^r$, which results in $B_n$ being spanned.
\end{proof}

\section{Lower bound for the symmetric line growth with $r=2$}\label{sec-line-2}

In this case, we have a lower bound that asymptotically matches
the bound from Lemma~\ref{lemma-line-sym-upper}. For $r>2$, we prove a weaker bound 
of the form $p^{-(r-1)+o(1)}$ in Section~\ref{sec-line-sym-lower}.

\begin{lemma}\label{lemma-line-lower-2} Assume $r=2$. 
If $t\ll {p^{-1}}\log\frac1p$, then $T\gtrsim t$. 
\end{lemma} 

\begin{proof}

A {\it neighborhood path\/} is a sequence 
$x_0, x_1,\ldots , x_n\in \mathbb Z^2$ of length $n$, starting at 
$x_0$ which makes neighborhood steps, that is, $x_{i}-x_{i-1}\in \mathcal N_0$, $1\le i\le n$. 
We say that the path makes a {\it turn\/} at $i$, $1\le i\le n-1$, 
if the vectors $x_{i}-x_{i-1}$ and $x_{i+1}-x_{i}$ are perpendicular.  
We call such a path {\it transmissive\/} if: 
\begin{enumerate}
\item[(1)] $x_0$ has two black sites in its neighborhood;  
\item[(2)] at every turn $i$, there 
is a black site on a horizontal or vertical line through $x_i$ within distance $2\rho$; 
\item[(3)] all black sites in (1) and (2) are distinct; and 
\item[(4)] all steps are of size $1$, and between any two turns the steps are 
all in the same coordinate direction and orientation.
\end{enumerate}
At any time, call a line {\it saturated\/} within a box at a given time 
if it contains two or more occupied 
sites within a neighborhood of one of its sites inside of the box.  
When a point becomes occupied at some time $t>0$, we 
call it {\it red\/}. Thus, every occupied point
at time $t$ is either red or black, but not both. 

\medskip
\noindent{\it Step 1\/}. 
Assume that there are no parallel saturated lines within distance $5\rho$ 
of each other at time $t$ for the dynamics restricted to 
$[-2\rho t,2\rho t]^2$, with $0$-boundary, by time $t$. Assume also that $\mathbf 0$ is occupied by time $t$. 
Then we claim that there exists 
a transmissive path ending at $0$ of length at most $\rho t$. 
\medskip
  
Assume that $x$ gets occupied at time $t>0$. We will first
prove by induction that there is a neighborhood path of length at most $t$, ending at $x$, 
that satisfies properties (1) and (2), but not necessarily (3) and (4), and is contained 
entirely in lines that are saturated within $[-2\rho t,2\rho t]^2$ by time $t-1$.

When $t=1$, the path consist of only the site $x$. We now proceed to proving the $t-1\to t$ 
induction step, for $t\ge 2$.
 
At time $t-1$, $x$ has two occupied neighbors, say $y_1$ and $y_2$, so that 
$y_1, y_2, x$ are all on the same line, and so that $y_1$ and $y_2$ are not both black. 
If, say, $y_1$ is black and $y_2$ is red, then we use the 
induction hypothesis on $y_2$, take the resulting neighborhood path that ends at $y_2$ 
and append the step from $y_2$ to $x$. It is possible that the so constructed path 
makes a turn at $y_2$, but then (2) holds.   

Now assume both $y_1$ and $y_2$ are red. By the induction hypothesis, there exist neighborhood 
paths that satisfy (1) and (2) and end at $y_1$ and $y_2$, with penultimate sites 
$z_1$ and $z_2$, with the proviso that. Here, we add the proviso that, if $y_1$ is occupied at time 
$1$, $z_1$ is one of the requisite black sites, and similarly for $y_2$. 
If, say, $x-y_1$ is parallel to the last step $z_1-y_1$ of the path to $y_1$,
we may complete the induction step by appending $x$ to the path to $y_1$. 
We proceed analogously if $x-y_2$ is parallel to the last step $z_2-y_2$. 
In all cases so far, the saturation claim holds for the new path to $x$.
The last possibility is that both $z_1-y_1$ and $z_2-y_2$ are perpendicular to 
$x-y_1$, and thus parallel. By induction, they 
lie on parallel saturated lines, and are within distance $2\rho$, in contradiction with 
our assumption. 

We thus constructed a path that satisfies (1) and (2), and we proceed 
to modify it to satisfy the other two requirements. We satisfy (4) by simply replacing 
all steps between two turns by steps of length 1 on the same line and with the same 
orientation. Now, if our path happens to make a turn at $i$, and $x_i=x_j$ for 
some $i<j$, we eliminate 
the loop $x_{i+1}, \ldots, x_j$ from the path. 
Similarly, we may eliminate a loop if $x_i=x_0$ for some $i>0$. (However, we cannot make the path
self-avoiding, as we cannot eliminate loops if two sites coincide at non-turns.)

Finally, suppose that the path makes turns at $i<j$, and 
$x_i$ and $x_j$ are at distance $5k$ or less. Then we have two parallel saturated 
lines within this distance, which contradicts our assumption. 
The same holds if there is a turn within distance $5\rho$ of $x_0$. 
It follows that (3) is also satisfied, and the claim in Step 1 is proved.

For future use, we consider two distinct red sites $x$ and $y$ that are occupied by time 
$t$. We have two transmissive paths $(x_i)$ and $(y_j)$, one to $x$ and the other to $y$.
If the path to $y$ makes a turn at some $j$ so that $y_j=x_i$ is on the path to $x$, then we 
replace to $x_0,\ldots,x_i$ with the path to $y_0,\ldots, y_j$. Analogously if we switch the 
roles of the two paths or if $x_0$ or $y_0$ are on the other path. After this, we say the two 
paths are {\it $m$-merged\/} if $x_i=y_i$ for $i\le m$.  In addition, 
we also declare them {\it $0$-merged\/} if $x_0$ and $y_0$ are 
within distance $5k\rho$. 

The result is that now any two turns that occur at different sites
use distinct black sites, and any turns use different sites than the 
initial sites. If the paths are not $0$-merged, the initial sites 
use different black sites as well. 


Take time $t=c(1/p)\log(1/p)$. We will choose the constant $c=c(\rho)$ to be 
small enough. Let $F_n$, $n\ge 0$,  be the event that there is a
transmissive path of length $n$ ending at $\mathbf 0$.

\medskip
\noindent {\it Step 2\/}. If $c$ is small enough, 
$
\prob{\cup_{n\le \rho t}F_n}\to 0$ as $p\to 0$.
\medskip

Dividing according to the number $i$ of turns this path makes, we get that, 
for some constant $C=C(\rho)$,
$$
\prob{F_n}\le  C\sum_{i=0}^{n-1} \binom{n}i (Cp)^i p^2
\le C (1+Cp)^{n} p^2
\le Ce^{Cnp} p^2.
$$
Therefore,
$$
\prob{\cup_{n\le \rho t}F_n}\le Cte^{Ctp} p^2\le Cp^{1-cC}\log(1/p), 
$$
which establishes the claim in Step 2.

Next we address the probability that there are two parallel saturated 
lines within distance $5\rho$ for the dynamics restricted to 
$[-2\rho t,2\rho t]^2$, with $0$-boundary, by time $t$. Consider the first such pair of 
lines. Each of these two lines must have, at the saturation time two occupied 
sites, at least one of which is black, within the neighborhood of a site. Therefore, there are 
two transmissive paths, one for each line, that each end at a site on the respective 
line, which has an additional black site within distance $2\rho$.  
Let $G_1$ be the event that the two paths do not merge for this first pair of 
parallel lines.

\medskip
\noindent{\it Step 3\/}. For $c$ small enough, $\prob{G_1}\to 0$.
\medskip

As black sites on all turns of both paths 
are distinct, 
$$\prob{G_1}\le  Ct(tp)^2 (te^{Ctp} p^2)^2=Cp (tp)^5e^{2Ctp}, $$
reflecting the selection of the pair of lines, chosing the position of the 
pair of occupied sites on each line, one of them being black, and then 
the choice of non-merged transmissive paths. The claim in step 3 follows. 


Finally, we consider the event $G_2$ that the two paths do merge
for the described first pair of 
parallel lines.

\medskip
\noindent {\it Step 4\/}.  For $c$ small enough, $\prob{G_2}\to 0$.
\medskip

Fix two sites $y$ and $z$, 
which do not lie on the same line, but are otherwise arbitrary. 
The key is to find an upper bound for the probability 
that there is a neighborhood path that satisfies (2)--(4) of length exactly $n$ connecting $y$ and $z$. 
Consider first 
the case when $n>||y-z||_1$. Then the path makes $i\ge 2$ turns, and 
there are some two successive turns at which the path makes 
a U-turn. If the locations of all other $i-2$ turns are
fixed, there are at most $2$ (symmetric) possibilities for locations of these two turns; 
all other locations change the length. So the upper bound on the probability is
\begin{equation}\label{eq-line-2-1}
\sum_{i=2}^{n} \binom n {i-2}i(Cp)^i\le Cp^2(np)e^{Cnp},
\end{equation}
reflecting the choice of the number of turns, position of the first U-turn, 
and positions of the black sites at each turn. 

Now consider the case $n=||y-z||_1$. Now the number of turns is $i\ge 1$, 
and  we merely observe that after the choice 
of first $i-1$ turns, there is at most one possibility for the last turn, which gives
the upper bound
\begin{equation}\label{eq-line-2-2}
\sum_{i=1}^{n} \binom n {i-1}(Cp)^i\le Cpe^{Cnp},
\end{equation}
as we can now only choose the positions of $i-1$ turns. 

If the first transmissive path makes $i$ turns, the second transmissive path
can join it at $i$ locations. By choosing which path is the second  one (i.e., the one that joins) 
we can also make sure that the initial point of the second path and 
the joining location are not on the same line. 
Putting all this together, we 
sequentially make the choices of: the pair of close parallel saturated lines; 
the pair of occupied sites, with one black, on each line; length of the first 
path; number $i$ of turns of the first path; positions of the $i$ turns on the first 
path; positions of black sites that make the first path transmissive; and locations
where the second path joins. For the probability bound that the second path 
contributes, we use  (\ref{eq-line-2-1}) and  (\ref{eq-line-2-2}). This gives the following 
upper bound:
\begin{align*}
\prob{G_2} & \le Ct (t^2 p^2)\cdot t \sum_{i=0}^{\rho t} \binom {\rho t} i  p^ip^2 i 
 \cdot (tp^2(tp)e^{Ctp}+ pe^{Ctp})
\\&\le Ct (t^2 p^2)\cdot t (pt)e^{Cpt} p^2
\cdot (tp^2(tp)e^{Ctp}+ pe^{Ctp})
\\&= C (tp)^5\cdot e^{2Cpt} 
\cdot ((tp)^{2}+ 1)\cdot p.
\end{align*}
This again goes to 0 if $c$ is small enough. 

From Steps 2--4, for small enough $c$, 
$$
\prob{\mathbf 0\text{ is occupied at time $t$}}\le \prob{(\cup_n F_n)\cup G_1\cup G_2}\to 0, 
$$
as desired.
 \end{proof}
 
 \begin{proof}
 [Proof of Theorem~\ref{thm-liner2}]
 The scaling of $T$ follows from Lemmas~\ref{lemma-line-sym-upper} and~\ref{lemma-line-lower-2}.
 
 We need to prove  the conclusion for the critical length, $L_c$. 
Lemma~\ref{lemma-line-sym-upper} gives the upper bound $L_c = \cO(p^{-1})$. On the other hand, if $n\ll p^{-1}$, the initial 
 configuration in $B_n$ with $0$-boundary is inert with high probability. It follows that $L_c=\Theta(1/p)$. 
  \end{proof}

\section{General upper and lower bounds for the line growth}\label{sec-line-general}

In this section we establish the possibly non-pure critical power for the line growth 
dynamics in which a site becomes newly occupied by having 
either $r$ horizontal or $s$ vertical occupied neighbors, with $2\le s\le r$. We denote 
$$
\gamma=\frac{(r-1)s}{r}, 
$$
the claimed critical power. 

\subsection{Upper bound in the asymmetric case} \label{sec-line-asym-upper}

In this subsection we assume that  $2\le s< r$.

\begin{lemma} \label{lemma-line-asym-upper}
Assume that $n\gg p^{- (\gamma+\epsilon)}$ for some $\epsilon>0$. Then there exists a large enough constant $C>0$  such that
the dynamics spans $B_n$ by time $Cn$. 
\end{lemma}



We consider the following comparison process, evolving in steps $\ell=1,2\ldots$. Divide the $n\times n$ box into
vertical strips of width $r$ and horizontal strips of width $s$. 
Call a vertical (resp.~horizontal) 
line {\it saturated\/} 
if it contains $s$ (resp.~$r$) occupied sites, all in the same horizontal 
(resp.~vertical) strip. Any step $\ell$ consists of two half-steps.  When $\ell=1$, the first 
half step occupies, in every vertical strip, the maximal contiguous interval of saturated vertical lines that 
includes the leftmost line. In the first half step for $\ell\ge 2$, in each vertical 
strip, consider for occupation only 
the leftmost vertical line that is not already occupied: if that line 
is saturated, make it occupied;
otherwise, that strip 
remains unchanged. In the second half-step, for any $\ell\ge 1$, we only consider the bottom-most 
unoccupied 
line 
in each horizontal strip: if that line is saturated, we occupy that line; otherwise, there are no changes in that strip. 
After the $\ell$th step, let $V^{(\ell)}_m$ be the number 
of vertical strips with at least $m$ occupied lines.  
Also, let $H^{(\ell)}$ be the number of occupied 
horizontal lines after the $\ell$th step.

We stop the process after completion of step $\hat\ell$ where 
$$
\hat\ell:=\min\{\ell:(\ell r+1)\epsilon\ge s/r\}.
$$


\begin{lemma}\label{lemma-asym-upper-ind}
For a small enough $\epsilon>0$ and for $\ell\le \hat\ell$,
\begin{align*}
&V^{(\ell)}_{r-2}\gep \frac{s}{r}+ (r-1)\epsilon\\
&V^{(\ell)}_{r-1}\gep \ell r\epsilon \\
&H^{(\ell)}\gep \ell r\epsilon +\epsilon +s-\frac{s}{r}-1.
\end{align*}
\end{lemma}

\begin{proof}
We prove that the lower power bound on $V^{(\ell)}_{r-2}$ holds at $\ell=1$, and 
the other two bounds are proved by induction on $\ell$.
As we perform only a bounded number of steps, we may (by a standard sprinkling argument) 
assume that the configuration 
of black sites is, at every half step, an independent product measure off 
already occupied lines. 
 For $\ell=1$, 
$V^{(1)}_m$ is with high probability larger than a constant times
$$
n(np^s)^m=n^{m+1}p^{sm}\gg p^{-((\gamma+\epsilon)(m+1)-sm)},
$$
whenever the power of $p$ is non-positive, 
and so
$$
V^{(1)}_m\gep s-\frac{s}{r}+\epsilon+m\left(-\frac{s}{r}+\epsilon\right).
$$
In particular, 
we get the desired two lower power bounds on $V^{(1)}_{r-2}$ and $V^{(1)}_{r-1}$. 
For $H^{(\ell)}$, we will consider help from 
 occupied vertical strips with at least $(r-1)$ occupied lines 
through the first half step of step $\ell$. 
Thus, we need a single black site next to the occupied $(r-1)$ vertical lines within a vertical strip to saturate 
a horizontal line. In particular, $H^{(1)}$ is with high probability at least a constant times
$$
p^{-\epsilon r}np\gg p^{-(\gamma+\epsilon+\epsilon r -1)},
$$
and so
$$
H^{(1)}\gep r\epsilon+\epsilon +s-\frac{s}{r}-1.
$$
For the $\ell\to \ell+1$ inductive step, we need to bound from below the 
probability that a fixed vertical strip that contains $r-2$ occupied vertical lines occupies an additional 
line, with the help of occupied horizontal lines. Using the lower bound on $H^{(\ell)}$ from the induction hypothesis and $s-1$ additional black points in a horizontal strip, we get a lower bound on this probability of at least a constant times 
$$
p^{-(\ell r\epsilon+\epsilon + s-\frac{s}{r}-1)+s-1}.
$$
The stopping rule $\ell\le \hat \ell$ ensures that the power of $p$ above is positive, so our lower bound 
for this probability is small.
This gives 
\begin{align*}
V^{(\ell+1)}_{r-1}\gep &\frac{s}{r}+(r-1)\epsilon
+\ell r\epsilon+\epsilon +s-\frac{s}{r}-1
-(s-1),
\end{align*}
which gives the desired lower power bound for $V^{(\ell+1)}_{r-1}$. We use this bound 
in the subsequent half step to bound $H^{(\ell+1)}$. Observe that the stopping rule $\ell\le \hat\ell$ guarantees
that $$(\ell+1)r\epsilon -1<\frac sr -1+(r-1)\epsilon<0,$$ 
so our lower bound on the probability 
that a fixed horizontal line gets occupied is small. This gives
\begin{align*}
H^{(\ell+1)}\gep & (\ell+1) r\epsilon+\gamma+\epsilon-1, 
\end{align*}
again confirming the induction step.
\end{proof}

\begin{proof}[Proof of Lemma~\ref{lemma-line-asym-upper}]
After our process stops, the probability of occupation of a vertical line in our comparison 
process is 
bounded from below by a strictly positive constant. It takes the original dynamics at most $Cn$ steps to occupy all vertices in the lines occupied by the comparison dynamics through step $\hat \ell$. The original dynamics, then, completely 
occupies a vertical strip in at most $n$ additional steps with high probability. 
Then the entire 
box $B_n$ is occupied in at most  $n$ additional steps. 
\end{proof}

\subsection{Lower bound in the asymmetric case}\label{sec-line-asym-lower}

In this subsection we again assume that $2\le s< r$ and $\gamma = (r-1)s/r$.

\begin{lemma} \label{lemma-line-asym-lower}
Assume that $p^{-\gamma+\delta}\ll n\ll p^{-\gamma}$. For small enough $\delta>0$,
the dynamics on $B_n$ with periodic boundary has $|\xi_\infty|=o(n^2)$ with probability 
converging to $1$.
\end{lemma}

We now call a vertical (resp.~horizontal) 
line {\it saturated\/} if it contains $s$ (resp.~$r$) occupied points within a neighborhood. We now, at
every step $\ell$, perform the following two half steps.
In the first half-step, 
we occupy all saturated vertical lines and in the second half-step, all horizontal saturated 
lines. Recall the definition of an $m$-clump from Section~\ref{sec-bp-lower}. After the $\ell$th step is completed, we let $V^{(\ell)}_m$ and $H^{(\ell)}_m$ be the numbers 
of vertical and horizontal clumps 
of size at least $m$, with $V^{(0)}_m=H^{(0)}_m=0$, for $m>0$ and 
$V^{(0)}_0=H^{(0)}_0=n$. 
Also let $\Delta V^{(\ell)}_m=V^{(\ell)}_m-V^{(\ell-1)}_m$ 
and  $\Delta H^{(\ell)}_m=H^{(\ell)}_m-H^{(\ell-1)}_m$.
 
\begin{lemma}\label{lemma-asym-lower-ind}
Assume that $n\ll p^{-\gamma}$.
For all $\ell\ge 1$,  
\begin{align*}
&V^{(\ell)}_m\lep m\left(-\frac sr\right)+s-\frac sr\\
&H^{(\ell)}_m\lep m\left(\frac sr-2\right)+s-\frac sr\\
&\Delta V^{(\ell)}_m\lep m\left(-\frac sr\right)+s-\frac sr+(\ell-1)\left(\frac sr-1\right)\\
&\Delta H^{(\ell)}_m\lep m\left(\frac sr-2\right)+s-\frac sr+(\ell-1)\left(\frac sr-1\right).
\end{align*}
Moreover, $V^{(\ell)}_{r-1}=0$ with high probability. 
\end{lemma}

\begin{proof} We proceed by induction. 
Observe that, as $s/r<1$, the upper power bounds for $\Delta V^{(\ell)}_m$ and 
$\Delta H^{(\ell)}_m$ are less than or equal to  the respective bounds for $V^{(\ell)}_m$
and $H^{(\ell)}_m$.
For $\ell=1$, the number of vertical $m$-clumps 
is bounded by a constant times 
$$
n(np^s)^m\ll p^{-\gamma(m+1)+sm}
$$
and therefore 
$$
V^{(1)}_m\lep m(\gamma-s)+\gamma=:\beta(m).
$$
To bound the number of horizontal $m$-clumps, we employ an  
optimization scheme described below.  

Similarly as in Section~\ref{sec-bp-lower}, 
we use FKG in the following fashion. Given a fixed (deterministic) set $\mathcal H$ 
of, say,  horizontal lines, condition
on exactly them being occupied by the procedure so far. By FKG, under this conditioning, 
the probability of any increasing event $E$ depending on the configuration 
outside $\mathcal H$ is smaller than the non-conditional probability of $E$.

Any horizontal $m$-clump consists of $m$ lines enumerated, say, 
from the lowest line in the clump. The line $i$ may be saturated by some vertical 
$a_i$-clump together with $r-a_i$ nearby black points on this line; here, 
$0\le a_i\le r-2$, since $V_{r-1}^{(1)}=0$ with high probability. 
Thus the number of horizontal $m$-clumps with helping numbers $a_1,\ldots, a_m$ is with high probability bounded 
above by a constant times 
$$
n\cdot p^{-\beta(a_1)}p^{r-a_1}\cdots p^{-\beta(a_m)}p^{r-a_m}.
$$
Let $a=a_1+\cdots +a_m$. Then we get 
\begin{align*}
H^{(1)}_m\lep\max_a(\gamma-mr+\gamma m+(\gamma-s+1)a).
\end{align*}
As $\gamma-s+1>0$, the maximum is achieved at $a=m(r-2)$ (that is, all $a_i=r-2$). 
It follows that 
\begin{align*}
H^{(1)}_m\lep\gamma-mr+\gamma m+(\gamma-s+1)m(r-2)=m\left(\frac sr-2\right)+\gamma=:\alpha(m).
\end{align*}
For the $\ell\to\ell+1$ induction step, let $\beta'(m)$ and $\alpha'(m)$ be the right-hand side of the 
upper power bounds for $\Delta V_m^{(\ell)}$ and $\Delta H_m^{(\ell)}$. 

To get a new vertical $m$-clump at step $\ell+1$, 
pick a $b$ and add $m-b$ vertical lines to an existing vertical 
$b$-clump, $0\le b\le m-1$. Let $a_1,\ldots,a_{m-b}$ be the sizes of helping horizontal clumps 
for lines to be saturated $1,\ldots, m-b$, counted from the left. The number of such new $m$-clumps is 
with high probability at most constant times
$$
p^{-\beta(b)}\cdot p^{-\alpha'(a_1)}p^{s-a_1}\cdots p^{-\alpha'({a_{m-b}})}p^{s-a_{m-b}}.
$$
Let $a=\sum_i a_i$. Observe that $a_i\ge 1$ for all $i$ since $\ell\ge 1$, so $a\ge m-b$. We get
\begin{align*}
\Delta V^{(\ell+1)}_m&\lep \max_{a,b}\:
b(\gamma-s)+\gamma-s(m-b) +a + a\left(\frac sr-2\right)
+(m-b)\left(\gamma+(\ell-1)\left(\frac sr-1\right)\right)\\
&=\max_{a,b}\: a\left(\frac sr-1 \right)+(m-b)(\ell-1)\left(\frac sr-1\right)-m\frac sr+\gamma\\
&=\max_{b} \: (m-b)\left(\frac sr-1 \right)+(m-b)(\ell-1)\left(\frac sr-1\right)-m\frac sr+\gamma\\
&=\max_{b} \: (m-b)\ell\left(\frac sr-1\right)-m\frac sr+\gamma\\
&= \ell\left(\frac sr-1\right)-m\frac sr+\gamma,\\
\end{align*}
as claimed. 

Proceeding to formation of horizontal clumps, we similarly get, with $0\le b\le m-1$ and 
$1\le a_i\le r-2$, $i=1,\ldots, m-b$, $a=\sum_i a_i$, and noting that we must use 
$\beta'({a_i})$ with $\ell$ replaced by $\ell+1$,
\begin{align*}
\Delta H^{(\ell+1)}_m&\lep 
\max_{a_i,b}\:\alpha(b)+\sum_{i}\beta'(a_i)+\sum_i a_i-r(m-b)\\
&=\max_{a,b}\: a\left(1-\frac sr \right)
+(m-b)\left(\gamma+\ell\left(\frac sr-1\right)-r-\frac sr+2\right)+m\left(\frac sr-2\right)+\gamma\\
&=\max_{b}\: (m-b)(r-2)\left(1-\frac sr \right)
+(m-b)\left(\gamma+\ell\left(\frac sr-1\right)-r-\frac sr+2\right)+m\left(\frac sr-2\right)+\gamma\\
&=\max_{b}\: (m-b)\ell\left(\frac sr-1\right)+m\left(\frac sr-2\right)+\gamma\\
&=\ell\left(\frac sr-1\right)+m\left(\frac sr-2\right)+\gamma,\\
\end{align*}
again verifying the inductive claim. 
\end{proof}

\begin{proof}[Proof of Lemma~\ref{lemma-line-asym-lower}] 
It follows from Lemma~\ref{lemma-asym-lower-ind} that, after a bounded number $\ell$ of steps, 
$\Delta V_1^{(\ell)}\lep 0$ 
and $\Delta H_1^{(\ell)}\lep 0$. Therefore the configuration after $\ell$ steps is with high probability inert. As
$V_1^{(\ell)}\lep\gamma-s/r<\gamma$, for small enough $\delta>0$, the final configuration occupies $o(n^2)$ sites, and Lemma~\ref{lemma-lb-easy}
applies.
\end{proof}



\subsection{The general lower bound for the symmetric case.}\label{sec-line-sym-lower}
In this section, we assume $s=r\ge 2$, so that $\gamma=r-1$. 

\begin{lemma} \label{lemma-line-sym-lower}
Fix an $\epsilon\in(0,1/2)$. 
Assume that $ p^{- (r-3/2-\epsilon)} \ll n\ll p^{- (r-1-\epsilon)}$. Then 
the dynamics on $B_n$ with periodic boundary has $|\xi_\infty|=o(n^2)$ with probability 
converging to $1$.
\end{lemma}

\begin{proof} 
Let $\gamma'=r-1-\epsilon$. 
We now at each step occupy all
$r$-saturated horizontal and vertical lines simultaneously. 
Then we have, as in the asymmetric case,  
$$
H_m^{(1)}, V_m^{(1)}\lep  m(\gamma'-r)+\gamma'=:\alpha(m)
$$
so that both numbers vanish with high probability for $m\ge r-1$. Now we claim that 
$$
\Delta H_m^{(\ell)}, \Delta V_m^{(\ell)}\lep  m(\gamma'-r)+\gamma'-(\ell-1)\epsilon=:\alpha'(m)
$$
for all $\ell\ge 1$. 

To prove the claim by induction, we again consider ways to get 
a new vertical $m$-clump at step $\ell+1$. We 
pick a $b$ and add $m-b$ vertical lines to an existing vertical 
$b$-clump, $0\le b\le m-1$, and let $a_1,\ldots,a_{m-b}$ be the helping horizontal clumps 
for line $1,\ldots, m-b$, as before. The number of such new $m$-clumps is 
with high probability at most constant times
$$
p^{-\alpha(b)}\cdot p^{-\alpha'({a_1})}p^{r-a_1}\cdots p^{-\alpha'(a_{m-b})}p^{r-a_{m-b}}.
$$
Let $a=\sum_i a_i$. Observe that $1\le a_i\le r-2$ for all $i$. We get
\begin{align*}
\Delta V^{(\ell+1)}_m&\lep \max_{a,b}\:
b(\gamma'-r)+\gamma'-r(m-b) +a + a\left(\gamma'-r\right)
+(m-b)\left(\gamma'-(\ell-1)\epsilon\right)\\
&=\max_{a,b}\: -a\epsilon-(m-b)(\ell-1)\epsilon+ m(\gamma'-r)+\gamma'\\
&=\max_{b}\: -(m-b)\epsilon-(m-b)(\ell-1)\epsilon+ m(\gamma'-r)+\gamma'\\
&=\max_{b}\: -(m-b)\ell\epsilon+ m(\gamma'-r)+\gamma'\\
&= -\ell\epsilon+ m(\gamma'-r)+\gamma',
\end{align*}
as claimed. The proof is concluded as that of Lemma~\ref{lemma-line-asym-lower}.
\end{proof}

\begin{proof}
[Proof of Theorem~\ref{thm-line}]
The conclusion follows from Lemmas~\ref{lemma-line-asym-upper} and~\ref{lemma-line-asym-lower}
for the case $s<r$ and from Lemmas~\ref{lemma-line-sym-upper}, and~\ref{lemma-line-sym-lower}
for the case $s=r$.
\end{proof}

\subsection{The perturbed line growth.}\label{sec-line-perturbed}

\begin{prop} \label{prop-line-perturbed} Assume $2\le s\le r$, and 
assume that we remove the $(r-1,s-1)$ square from the line-growth 
Young diagram of Theorem~\ref{thm-line} 
to get the zero-set with minimal counts $(0,s)$, $(r,0)$ and 
$(r-1,s-1)$.  Then $\gamma_c=(r-1)s/r$ remains the same as in Theorem~\ref{thm-line}. 
\end{prop}

\begin{proof}  We will address the asymmetric case $s<r$ first, proving that 
the lower bound in Lemma~\ref{lemma-line-asym-lower} still 
holds.
We join the proof of Lemma~\ref{lemma-asym-lower-ind}, with the corresponding definitions of $\alpha$ and $\beta$, and 
claim that the final configuration is still with 
high probability inert, which clearly suffices to verify this claim. 

Fix $k\ge 1$ and let $N$ be the number of 
sites in the final configuration in the proof of Lemma~\ref{lemma-asym-lower-ind} that have
at least $k$ occupied neighbors. The neighborhood of
such a site must intersect some $a$ occupied horizontal lines and some $b$ occupied vertical lines
and contain additional $k-a-b$ black sites. Therefore, $N$ is with high probability 
much smaller than
$$
p^{-\alpha(a)} p^{-\beta(b)} p^{k-a-b},
$$
and then
\begin{align*}
N&\lep \max_{0\le a, 0\le b\le r-2}a\left(\frac sr-1\right)
+b\left(1-\frac sr\right)+2\gamma-k
\\&=
(r-2)\left(1-\frac sr\right)+2\gamma-k
\\&=r+s-2-k,
\end{align*} 
and so $N=0$ with high probability as soon as $k\ge r-s-2$, which 
covers exactly our case.

To address the symmetric case, we prove that the lower bound in Lemma~\ref{lemma-line-sym-lower} still 
holds, again joining its proof, with the corresponding definition of $\alpha$. 
We observe that the procedure in the proof produces the final configuration in which there 
are no $(r-1)$-clumps. We claim that this configuration is still with high probability 
inert. To prove this claim, let $k\ge 1$ and let $N$ be the number of 
sites in this configuration with $k$ occupied neighbors. The neighborhood of
such a site must intersect some $a$ occupied horizontal lines and some $b$ vertical lines
and contain additional $k-a-b$ black sites. Therefore, $N$ is with high probability 
much smaller than
$$
p^{-\alpha(a)} p^{-\alpha(b)} p^{k-a-b},
$$
and then
\begin{align*}
N\lep\max_{0\le a,b \le r-2}\left[-\epsilon(a+b)+2\gamma'-{k}\right]=2\gamma'-k, 
\end{align*} 
and so $N=0$ with high probability as soon as $k\ge 2\gamma'$, which 
holds when $k=2(r-1)$.
\end{proof}

\section{Finite L-shaped Young diagrams} \label{sec-L-finite}

It follows from the results of the Theorems~\ref{thm-bp} and~\ref{thm-line} that any zero-set that
lies between the threshold-$r$ bootstrap percolation triangle and the $r\times r$ box has 
the the upper and lower powers that are both $r-\mathcal O(\sqrt r)$, for large $r$. 
In this section, we consider
some examples of zero-sets that do not satisfy this restriction.
Our main focus will be symmetric L-shaped Young diagrams, 
but we will also briefly address the asymmetric ones. We begin with a lemma that illustrates our methods for the simplest $L$-shaped Young diagrams, and provides a lower bound used in Lemma~\ref{lemma-L4}.

\begin{lemma} \label{lemma-L1} Assume $\cZ$ has minimal counts $(0,r)$, $(r,0)$ and 
$(1,1)$ for $r\ge 2$. Then $\gamma_c=r/2$.
\end{lemma}

\begin{proof}
As $r$ diagonally adjacent black sites occupy an $r\times r$ square in $r$ steps. If $n\gg p^{-r/2}$, $B_n$ contains such an $r\times r$ box with high probability, so $\gamma_c\le r/2$. 

To prove the lower bound, consider, as usual, the dynamics on the box $B_n$ with periodic boundary. 
Assume $n\ll p^{- r/2}$. 
Choose some sequence $M$ such that $1\ll M^r\ll 1/(p^{r/2}n)$. 

Call a line {\it saturated\/} if it contains $r$ occupied 
sites within an interval of length $M$. Let $\tau$ be the first time 
$B_n$ contains a saturated line. If a site $x$ is occupied for the first time 
at time $t\in [1,\tau]$, 
then both $\cN_x^h$ and $\cN_x^v$ must contain an occupied site at time $t-1$. 
Each of these occupied sites is either black or becomes occupied at some previous time $u\in[1,t-2]$.
Continuing backwards in time, we find that on each line through
$x$ there is a black site within distance $(r-1)\rho$, as we cannot continue for 
more than $r-1$ steps, or else a line through $x$  would be saturated before time $\tau$. If $L$ is a line 
that becomes saturated at time $\tau$, and $x_1,\ldots, x_r$ are 
the occupied sites on $L$ at time $\tau$ within distance $M$, then, for each $i$, 
either $x_i$ is black or there is a black site within distance $(r-1)\rho$ 
on the line through $x_i$,
perpendicular to $L$. Therefore, if $B_n$ ever contains a saturated 
line, there is an
$M\times M$ box within $B_n$ that contains $r$ black sites, and this 
happens with probability at most $M^{2r}n^2p^r\to 0$.  If the box never contains a saturated line,
$$
|\xi_\infty|\le n\cdot (n/M+1)(r-1)\ll|B_n|, 
$$
and Lemma~\ref{lemma-lb-easy} implies that $\gamma_c\ge r/2$.
\end{proof}

\begin{lemma}\label{lemma-L2}
Assume $\cZ$ has minimal counts $(0,r)$, $(r,0)$ and 
$(2,2)$ for $r\ge 6$. Then $\gamma_c=2r/3$.
\end{lemma}

\begin{proof} To prove the upper bound assume first that $n\gg p^{- 2r/3}$ 
and consider the dynamics on $B_n$. 
This box will be spanned by time $Cn$ if the following event happens: 
there are two neighboring horizontal 
lines in $B_n$ that each contain an interval of $r$  black sites, and next to this pair of lines there 
is a strip of $r-2$ additional horizontal lines each of which contains $2$ neighboring 
black sites. This happens with high probability  as $n(np^r)^2\gg 1$ and 
$np^2\gg p^{-( 2r/3-2)}\gg 1$. 
It follows that $\gamma_c\le 2r/3$.

To prove the lower bound, assume $n\ll p^{- 2r/3}$. We will consider 
the slow version of the dynamics on $B_n$ with periodic boundary, in which 
we occupy at every time step a single site that can be occupied, chosen arbitrarily, provided it exists. 

We will often refer to sites and lines as being {\it near\/} each other or 
{\it nearby\/}. Within the context of this 
proof, this means that that the distance between objects in question is at most 
$Cr\rho$, for some suitable large constant $C$.  We will choose $M$ so that 
$1\ll M\ll 1/(p^{2r/3}n)^\epsilon$, for any $\epsilon>0$. For a line $L$ and 
a set of sites $S$, we call $S$ {\it $M$-near\/} $L$ if: all sites of $S$ are near 
$L$; and the projection of $S$ to $L$ fits into an interval of length $M$. 

Again, a {\it saturated\/} line contains $r$ occupied sites within an interval of length $M$, which we call the {\it saturation sites\/} of this line.
(Note that they do not need to be near each other as $M$ is much larger than $Cr\rho$.)
Upon saturation, we declare the line completely occupied. If a line is saturated 
initially it may have more than one choice of $r$ saturation sites in which case we just 
make an arbitrary selection of them. If $x$ is a saturation site on a line 
$L$, any black site near $x$ on a line perpendicular to $L$ through $x$ is  
{\it associated\/} to $x$.

Our strategy is to prove that saturated parallel lines likely remain at distance of order $M$.
To this end, let time $\tau_2$ be the first time when there are two parallel 
saturated lines within distance $3M$. Assume that $\tau_2<\infty$. 
Denote the resulting lines by $L_1$ and $L_2$, and assume they are 
horizontal and that $L_1$ was saturated first, at time $\tau_1<\tau_2$. The main problem that 
we face in the rest of the argument is that there may be isolated vertical or horizontal lines
that get saturated before time $\tau_2$ and help in saturation of $L_1$ and $L_2$
(see Figure~\ref{fig-L2}). 
We claim that one of the following must hold: 
\begin{enumerate}
\item [(1)] there are at least $2r-2$ black sites in some $5M\times 5M$ box inside $B_n$; or 
\item [(2)]
the following occur disjointly: 
\begin{itemize}
\item either there are $r$ black sites $M$-near $L_1$, or there are $r-1$ black 
 sites $M$-near $L_1$
{\it and\/} a line near them perpendicular to $L_1$, which disjointly has  
$r-1$ $M$-near black sites; and
\item the same is true for $L_1$ replaced by $L_2$. 
\end{itemize}
\end{enumerate}

If a site $x$ on $L_1$ gets occupied at time $t\in [1,\tau_1]$, and it does 
not lie on a saturated vertical line at time $t$,  
then it must have two previously occupied 
sites on the vertical part of its neighborhood; if at 
least one of them is non-black, we can find two previously occupied sites 
in the vertical part of its neighborhood, and so on. Therefore, if such an $x$ is a
saturation site for $L_1$, it is either black or there are 
two black sites associated to $x$. Similarly, every saturation site $y$ on $L_2$
that is occupied at time $t\le \tau_2$ and does not
lie on a saturated vertical line at time $t$
is either black or has one black site off $L_1$ associated to it.

Assume first that a saturation site 
on $L_1$ and a saturation site on $L_2$ 
lie on nearby vertical lines. We claim that in this case (1) holds. 
Note that then all saturation sites 
on $L_1$ and $L_2$ 
are within horizontal distance $3M$ and therefore intersect at most 
one vertical line that is saturated by time $\tau_2$; this follows by minimality of $\tau_2$. 
This vertical line contains 
at most two saturation sites for $L_1$ and $L_2$, one on each of these two horizontal 
lines. Now consider all vertical lines through 
saturation sites on $L_1$ and $L_2$ that are {\it not\/} saturated at time $\tau_2$.  
Assume that $a\in [0, r]$ of these lines 
contain two saturation sites, one on $L_1$ and one on $L_2$.  Note that 
each of these $a$ lines must contain two black sites near
$L_1$. If $a=r$, then (1) clearly holds. If $a\le r-1$, there
are a least $2(r-1)-2a$ remaining lines, each 
of which contains at least one black site near $L_1$ or $L_2$,  and consequently these black sites are within vertical distance $3M$ of each other. 
This produces $2(r-1)$ black sites that satisfy (1).  

Now assume that no vertical line through a saturation site
on $L_1$ is near a vertical line through a saturation site
on $L_2$.  
Then each of the two lines $L_1$ and $L_2$ may have zero or one saturation site included
in a previously saturated vertical line. 
If neither have such a vertical line, then clearly (2) occurs.

If a vertical line $L_3$ covers a 
saturation site of $L_1$ and is saturated before time $\tau_1$, then there are
$r-1$ saturation 
sites for $L_1$ not on $L_3$; they are either all black or else there 
are at least $r$ associated black sites for $L_1$ not on $L_3$. In the 
former case, we have two additional possibilities. The first possibility is that a saturation site 
for $L_3$ is near a saturation site for $L_1$, in which case we use the fact that 
at least $r-1$ saturation 
sites, and therefore at least $r-1$ associated black sites, for $L_3$ are not on $L_1$
and (1) occurs. 
The second possibility is that no 
saturation sites for $L_1$ and $L_3$ are nearby in which case $L_3$ has at least $r-1$ associated 
black sites, disjoint from the $r-1$ black sites on $L_1$.  

Now we look at $L_2$, whose saturation sites are now horizontally far away from those for $L_1$. 
If there is no vertical saturated line at time $\tau_2$ that covers one of the 
saturation sites for $L_2$, we have $r$  black sites $M$-near $L_2$, disjoint 
from those for $L_1$ and $L_3$ (if it exists) in the previous paragraph. 
Otherwise, let $L_4$ be such a line. Then
$L_2$ has $r-1$ associated black sites not on $L_1\cup L_4$. If a saturation site
for $L_4$ is within distance $3M$ of $L_2$, then it is possible that one saturation site for $L_4$ 
is on $L_1$ and we either have: $r$ associated black sites for 
$L_4$ (which are all within distance $4M$ of $L_2$); or $r-1$ black saturation sites on $L_4$ ---
in either case (1) occurs. The last possibility (see Figure~\ref{fig-L2}) is that all saturation sites 
for $L_4$ are at least distance $3M$ from $L_2$, in which case it is possible that another previously
saturated line (not $L_1$) covers one of them, which means that we have at least $r-1$ associated 
black sites, disjoint from all others found so far. It follows that (2) holds and the 
claim is established. 

From the claim, we get that 
\begin{equation}
\begin{aligned}\label{eq-L2-1}
\prob{\tau_2<\infty}  &\le CM^{4r}\left[n^2p^{2r-2}+n(np^r)^2+n  (np^{r-1})^2 np^r + n (np^{r-1})^4\right]
\\&=  CM^{4r}\left[n^2p^{2r-2}+n^3p^{2r}+n^4p^{3r-2}+n^5p^{4r-4}\right]\\
&\to 0,
\end{aligned}
\end{equation}
provided that $2r/3\le r-1$, $8r/3\le 3r-2$, and $10r/3\le 4r-4$, which all hold for $r\ge 6$.

If $\tau_2=\infty$, then we have at most $n/M$ saturated lines in $\xi_\infty$, so that 
$$
|\xi_\infty|\le (n/M)\cdot n+n\cdot (n/M+1)(r-1)\ll |B_n|,
$$
thus Lemma~\ref{lemma-lb-easy} implies that $\gamma_c\ge 2r/3$.
\end{proof}

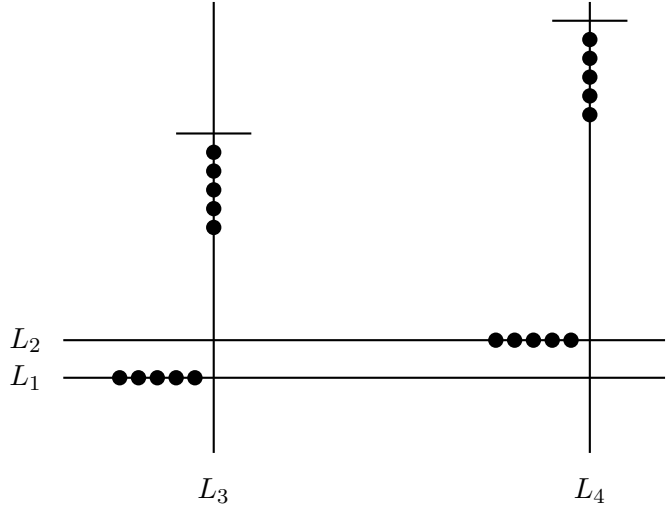
\begin{figure}
\begin{center}
\begin{tikzpicture}
 \draw [color=black,thick] (0,3)--(8,3);
 \draw [color=black,thick] (0,3.5)--(8,3.5);
 \node at (-0.5,3) {$L_1$};
 \node at (-0.5,3.5) {$L_2$};
 \draw [color=black,thick] (2,2)--(2,8);
 \draw [color=black,thick] (7,2)--(7,8);
 \node at (2,1.5) {$L_3$};
 \node at (7,1.5) {$L_4$};
 
 \foreach \y in {0,1,2,3,4}
     \fill [black] (0.75+0.25*\y,3) circle (2.9pt);
 \foreach \y in {0,1,2,3,4}
     \fill [black] (5.75+0.25*\y,3.5) circle (2.9pt);
 \foreach \y in {0,1,2,3,4}
     \fill [black] (2,5+0.25*\y) circle (2.9pt);
  \foreach \y in {0,1,2,3,4}
     \fill [black] (7,6.5+0.25*\y) circle (2.9pt);
     \draw [color=black,thick] (1.5,6.25)--(2.5,6.25);
 \draw [color=black,thick] (6.5,7.75)--(7.5,7.75);
 \end{tikzpicture}
\end{center}
\caption{An illustration of one case from the proof of Lemma~\ref{lemma-L2}, with $r=6$ of the configuration of black sites leading to the last term on the first line 
of (\ref{eq-L2-1}). One of the saturation sites on $L_1$ is generated by the previously saturated line 
$L_3$, which, in turn, may have one of its saturation sites generated by a previously saturated line. 
Analogously for $L_2$.}
\label{fig-L2}
\end{figure} 

The restriction $r\ge 6$ is likely not necessary. Indeed, for $r=3$ we have a separate proof
that $\gamma_c=2$ (Proposition~\ref{prop-line-perturbed}), while for $r=4,5$ the argument for the lower bound 
can possibly be extended (starting with the 
two previously saturated lines that contribute to saturation sites on $L_3$ and $L_4$ in Figure~\ref{fig-L2}), which 
we however did not pursue.

\begin{lemma}\label{lemma-L3}
Assume $\cZ$ is given by the minimal counts $(0,r)$, $(s,s)$, and $(r,0)$, with
$1\le s<r$. 
Then a lower and an upper power are, respectively,
$$\gamma_\ell=
 \frac{\hm(r-\hm+1)}{1+\hm}, \text{ where }\hm=\min(\lfloor(-1+\sqrt{4r+9})/2\rfloor,\lfloor s/2\rfloor),
$$
and
$$
\gamma_u= \frac{rs}{s+1}.
$$ 
%
\end{lemma}
 
\begin{proof}
For the upper power, consider the dynamics on $B_n$ with $n\gg p^{-\gamma_u}\ge p^{-s}$. 
Every fixed line 
is then likely to contain an interval of $s$ black sites. Moreover, a horizontal strip of width $s$ with each of 
its horizontal lines containing an interval of $r$ black sites is also likely 
to exist within $B_n$, as 
$n(np^r)^s=n^{s+1}p^{rs}\gg 1$. With high probability, then, there is such a strip with $r-s$ neighboring horizontal 
lines, such that each contains an interval of $s$ black sites. This configuration causes complete occupation of $B_n$ 
by time $3n$.

For the lower power, assume $n\ll p^{-\gamma_\ell}$.
adopt the same slow version of the dynamics as in the previous proof, with 
$M$ such that $1\ll M\ll 1/(np^{-\gamma_\ell})^\epsilon$ for every $\epsilon>0$, and the same definitions of saturated lines and nearby points. Pick an integer $m\le s/2$.  

Let $\tau$ now be the first time the dynamics produces $m$ saturated parallel lines within distance 
$3rM$. We claim that $\prob{\tau<\infty}\to 0$, for a judicious choice of $m$. Assume that $\tau<\infty$ and 
that the $m$ lines produced at time $\tau$ are horizontal. 
At time $\tau$, each occupied site on one of these lines either: is on a previously 
saturated vertical line;
is a black site; or has $s-m+1$ black sites nearby on the vertical line through it. 

Divide the $m$ horizontal
lines into $k$ groups of $a_1,\ldots, a_k$ lines so that, for two lines $L_1$ and $L_2$ in 
different groups, no vertical line through an occupied site on $L_1$ is near a
vertical line through an occupied site on $L_2$. Here $a_1+\ldots+a_k=m$ and $1\le a_i\le m$.
We now proceed to finding an upper bound for the probability for a fixed $m$ and a fixed division.
 
Once we fix the vertical locations of the $m$ lines, the 
first group, say, yields the probability bound a power of $M$ times $np^{a_1(r-m+1)}$, as:
every occupied site on these lines that does not lie on a vertical saturated line 
is either black or has  $s-a_1\ge a_1$ nearby black sites on its vertical line; and 
there are at most $m-1$ saturated vertical lines that can help, by minimality of $\tau$. 
This results in the bound a power of $M$ times 
$$n\cdot n^kp^{(a_1+a_2+\cdots+a_k)(r-m+1)} =n^{1+k}p^{m(r-m+1)}\le n^{1+m}p^{m(r-m+1)}.$$
The proof of the lower bound now reduces to the routine verification that the function 
$$
\frac{m(r-m+1)}{1+m}=r+3-\left(1+m+\frac{r+2}{1+m}\right), 
$$
achieves its maximum for $m\in [1,\lfloor s/2\rfloor]$ at $m=\hm$. 
\end{proof}

\begin{lemma}  \label{lemma-L4}
Assume $\cZ$ is given by the minimal counts $(0,r)$, $(s,1)$, and $(r,0)$, with
$1\le s<r$. 
For $s\le r/2$ we have pure critical 
power $\gamma_c=r/2$. For $s>r/2$, 
we have lower and upper powers $s-\frac{s}{r}$ and  $s+1-\frac{2s+1}{r+1}$. 
\end{lemma} 
 
\begin{proof}
We only need to provide the  
upper bounds, as lower bounds follow from Lemmas~\ref{lemma-L1} and~\ref{lemma-line-asym-lower}. 
As usual, we consider the dynamics on $B_n$. 
Observe that a horizontal strip of width $r$, which contains 
a horizontal line with an interval of $r$ black sites and $r-1$ lines, 
all with an interval of $s$ black sites, will occupy $B_n$ by time $(r+1)n$.

When $s\le r/2$, and $n\gg p^{-r/2}$, 
such a strip exists with high probability, as $n(np^r)\gg 1$ and $np^s\gg 1$. 
When $s>r/2$ and $n\gg p^{-( s+1-{(2s+1)}/{(r+1)})}$, the same holds as
$$n(np^r)(np^s)^{r-1}=n^{r+1}p^{r+s(r-1)}\gg 1.$$
\end{proof}

\begin{proof}[Proof of Theorem~\ref{thm-L-finite}]
The four parts are established by the four lemmas in this section.
\end{proof}

\section{Infinite L-shaped Young diagrams}\label{sec-L-infinite}

In this section, we consider the infinite Young diagrams with two minimal counts
$(s_1,s_2)$ and $(0,r)$,  
with $r\ge 2$, $1\le s_2\le r-1$, and $s_1\ge 1$. We let 
$$
\gamma=\frac{rs_1+s_2}{1+s_1}, 
$$
so that our plan is to prove $\gamma_c=\gamma$. We will make use of the dynamics on a rectangular
box $B$ of dimensions $A\times Ap^{-s_2}$, for a suitably chosen $A$ that increases as a power 
of $p$. 

\begin{lemma}\label{lemma-L-inf-ub} The upper bound  
$T\lesssim p^{-\gamma}$ holds. Hence $\gamma_c\le \gamma$. 
\end{lemma}
\begin{proof} 
Assume that 
$A\gg p^{-(r-s_2)s_1/(1+s_1)}=p^{-(\gamma-s_2)}$. Let 
$G_1$ be the event that $B$ contains 
$s_1$ neighboring vertical 
lines, each of which contains an interval of $r$ black sites. The event $G_1$ occurs with high probability because
$$
A(Ap^{-s_2}p^r)^{s_1}\gg 1.
$$
On $G_1$, we get an occupied vertical strip of width $s_1$ in $Ap^{-s_2}$ time steps. 

\begin{figure}[ht!]
\begin{center}
\begin{tikzpicture}

\fill[lgrey] (0.75,3) rectangle (1,4);
\fill[lgrey] (1,6) rectangle (1.25,7);
\fill[lgrey] (1.25,2) rectangle (1.5,3);
\fill[lgrey] (1.5,4) rectangle (1.75,5);
\fill[lgrey] (1.75,0) rectangle (2,1);
\fill[lgrey] (2,5) rectangle (2.25,6);

\draw (0,0) grid [xstep=.25,ystep=1] (2.25,8);
\draw [color=purple, ultra thick] (2.125,1.5)--(2.125,5.5)--(1.875,5.5)--(1.875,0.5)--(1.625,0.5)--
                        (1.625,4.5)--(1.375,4.5)--(1.375,2.5)--(1.125,2.5)--(1.125,6.5)--
                        (0.875,6.5)--(0.875,3.5);
\draw [thick, ->] (2.7,1.4) -- (2.125,1.5);
\node [right] at (2.7,1.4) {$\zeta_0$};

\draw [decorate,decoration={brace}, thick] (0.75,-0.1)--(0,-0.1);
\node at (0.375,-0.42) {$S$};

\fill[dgrey] (0,0) rectangle (0.75,8);
\draw[thick] (0,0) grid [xstep=.25,ystep=8] (0.75,8);
 
\end{tikzpicture}
\end{center}
\caption{Transmission of occupation to the cell $\zeta_0$. The dark grey 
strip $S$ has a vertical interval of $r$ black sites in each of its $s_1$ 
(which is set to 3 for the purpose of this illustration) columns, 
which all get occupied. Light grey columns have thickness $1$ and are cells that each include
a vertical interval of $s_2$ black sites. All cells on the purple path get 
fully occupied, starting with the one next to $S$.} 
\label{fig-Linf-super}
\end{figure}
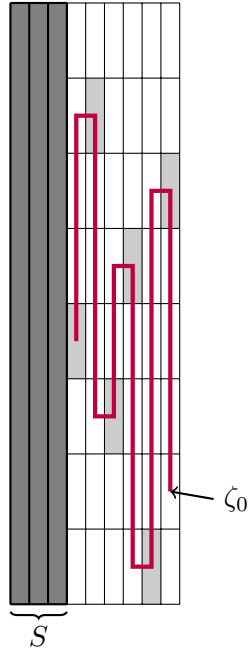

Now divide $B$ into vertical strips of height $h=Cp^{-s_2}$ and width $1$, which we call 
{\it cells\/}. We call such a cell 
{\it good\/} if it contains a vertical subinterval of length $s_2$ full of black sites. 
The probability $P_g$ that a fixed cell is good can be made arbitrarily close to $1$ if $C$ is 
chosen to be sufficiently large and then $p$ small enough. 
Connect every good cell with the two horizontally adjacent cells (good or not) and 
connect \underline{any} two vertically adjacent cells (again, good or not). 
Assume that $G_1$ happens, resulting in the occupied strip $S$
by time $T_1$. Then, any cell to the right (resp., left) of $S$ that is 
connected through a path of connected cells of length $k$ to the left (resp., right) edge of $B$ is 
completely occupied by additional time $hk$ (see Figure~\ref{fig-Linf-super}). 

Let $G_2$ be the event that the top and the bottom halves of 
every column both contain a good cell. As 
$$\prob{G_2^c}\le 
A(1-P_g)^{A/(2C)}\le A\exp(-P_gA/(2C))\to 0,$$ 
$G_2$ happens with high probability.

 Fix a cell $\zeta_0$. If $\zeta_0$ is below (resp.~above) the middle line of $B$, 
find the closest good cell $\zeta_1$ above (resp.~below) the midline in the column of $\zeta_0$; for the 
purposes of this part of the argument, we allow cells outside $B$, but this will not happen on the event $G_2$. 
Then find the first good cell $\zeta_2$ below (resp.~above) the midline in the column to the left of 
$\zeta_1$, then a good cell $\zeta_3$ in the column to the left of 
$\zeta_2$, with the same above-below convention, etc., until reaching a cell in the vertical line that contains the 
leftmost column of $B$. This constructs a path of 
connected cells whose length is bounded by $A$ plus twice the sum of at most $A$ independent Geometric($P_g$) random 
variables, where the additional $A$ accounts for the initial portion of the path from $\zeta_0$ to the midline. The probability that this length exceeds $5A$ for any fixed 
cell $\zeta_0$ is at most $\exp(-cA)$ for some constant $c>0$. Let $G_3$ be the event that 
the length of the path is at most $5A$ for {\it every\/} cell $\xi_0$.
As the number of cells within $B$ is bounded by $A^2$, 
$$
\prob{G_3^c}\le A^2\exp(-cA). 
$$
The event $G_4$ is the mirror image of the event $G_3$: every ``left'' is replaced by ``right'' and vice versa, so $\prob{G_4}=\prob{G_3}$.
If $G_1\cap G_2\cap G_3\cap G_4$ happens, then box $B$ gets fully occupied by time a constant 
times $Ap^{-s_2}$. Finally, by FKG, $\prob{G_1\cap G_2\cap G_3\cap G_4}\ge 
\prob{G_1}\prob{G_2}\prob{G_3}\prob{G_4}\to 1$.
\end{proof}

It remains to prove of the matching lower bound, thus we 
assume for the rest of this section that $A\ll p^{- (\gamma-s_2)}$. Again, the term {\it 
nearby\/} will mean within $\ell^\infty$-distance $Cr\rho$ for a suitable large constant $C$.
We start by two straightforward observations.
First, with high probability, 
the event that
\begin{itemize}
\item[(E1)] there are at most $s_1-1$ nearby vertical 
lines such that, disjointly, each line has $r$ nearby black sites that are nearby each other
\end{itemize}
happens, as now $
A(Ap^{-s_2}p^r)^{s_1}\ll 1.
$ 
Also,  
$$
|B|=\frac{A^2}{p^{s_2}}\ll p^{- \frac{2rs_1-s_1s_2+s_2}{1+s_1}},
$$
and so with high probability the event that
\begin{itemize}
\item[(E2)] there exists no set of nearby black sites of size 
$$
\left\lceil\frac{2rs_1-s_1s_2+s_2}{1+s_1}\right\rceil\le 2r
$$
\end{itemize} 
also happens. From now on, we assume that the two events (E1) and (E2) 
happen. 


Starting with the black sites, we perform the following two-stage occupation process, in which 
we call the sites that get occupied in the first 
step {\it red\/} and those that get occupied in the second step {\it blue\/}. 
Declare a vertical line $L$ to be {\it saturated\/} if it contains $r$ nearby sites, each of 
which is either black or is covered by a horizontal line that contains two 
nearby black sites that are also near $L$. We call such $r$ sites on $L$ 
{\it saturating\/} sites. In the first step we 
occupy, i.e., make red, all non-black points on all saturated vertical lines.

\begin{lemma}\label{lemma-L-inf-lb0}
No two 
saturating sites, for two different nearby vertical lines,
can be on the same horizontal line. Therefore, two different vertical lines 
cannot share black sites that make them saturated. 
\end{lemma}

\begin{proof}
Assume that $L_1$ and $L_2$ are two nearby saturated lines, and that
$a\ge 1$ is the number of horizontal lines that contain saturating sites 
of both $L_1$ and $L_2$, which together contain at least $2a$ nearby black sites. Further, 
the lines through the remaining $2(r-a)$ saturating sites 
($r-a$ on each of $L_1$ and $L_2$) together contain at least $2(r-a)$ black sites. This 
results in at least $2a+2(r-a)=2r$ nearby black sites, violating (E2).
(Here we used $a\ge 1$ only to force these black sites to be nearby.)
\end{proof}

In the second stage, we use the dynamics with zero-set determined by the 
single minimal count $(s_1,s_2)$, which is a critical dynamics \cite{BBMS1}. 
(Note that this is exactly the 
classic modified bootstrap percolation \cite{Hol}
when $\rho=1$ and $s_1=s_2=1$.) We start  
from the 
occupied set consisting of red and black sites. 
We use the slow 
version of this process, in which we occupy, that is, make blue, exactly one site that can 
be occupied, 
chosen arbitrarily, per time 
step, until we reach an inert configuration. Call a {\it cluster\/} a maximal connected 
set of black and blue sites, where we use the graph induced by our range $\rho$ 
neighborhood $\cN_x$ for connectivity. 
The {\it hull\/} of a cluster is the smallest rectangle that contains it. Note that hulls
of different clusters may intersect. We call a black site $x$ {\it enhanced\/} if it has 
$s_1$ red or black sites (including itself) within a horizontal interval of length $2\rho+1$ 
including $x$.

\begin{lemma}\label{lemma-L-inf-lb1} At any step of the second stage dynamics, 
every blue site has an enhanced black site in the
row of its hull, and $s_2$ black sites within a vertical interval of 
$2\rho+1$ sites in the column of its hull.
\end{lemma}

\begin{proof} We argue by induction: this property trivially holds initially, and then 
is maintained by addition of each blue point. Indeed, assume $x$ gets painted blue 
at some step. To prove the claim about the row of $x$, note that 
$x$ must have $s_1$ 
already colored horizontal neighbors, not all of which can be red by (E1) and 
Lemma~\ref{lemma-L-inf-lb0}. If at least one is 
blue, then the claim follows by induction; otherwise, one of them is black and thus 
is enhanced. To prove the claim about the column of $x$, note that $x$ must have 
$s_2$ already colored vertical neighbors, but now none can be red. If at least one is blue, the
claim follows by induction, and if all are black the claim follows immediately.
\end{proof}

\begin{lemma} \label{lemma-L-inf-lb2}
The second-stage dynamics never produces a hull with either dimension exceeding $b_0=2r\rho$.
\end{lemma}

\begin{proof}
Assume that $b$ is the largest dimension of a hull at time $t$, the first time 
the claim is violated. Such a hull 
contains at least $\lceil b/\rho\rceil$ black sites, by Lemma~\ref{lemma-L-inf-lb1}. 
Therefore, the claim holds initially, before any blue sites 
are created. Assume now that $t>0$ and let $x$ be the site 
that becomes blue at time $t$. As $x$ is connected 
to at most $4$ clusters at step $t-1$, 
$$2r\rho\le b\le 4\cdot 2r\rho+2\rho+1\le 11r\rho.$$ 
It follows that we have at least $2r$ black 
sites in a $11r\rho\times 11r\rho$ box, contrary to (E2).
\end{proof}


It follows from the above two lemmas that, for every blue site $x$, there are $s_2$ black sites 
in the vertical line through $x$ within an interval of size $2\rho r$ containing 
$x$. Also, there are at least 
$s_1$ red or black sites, at least one of which is black, on the horizontal line 
through $x$ within an interval of the same size. This will enable us to prove the 
key property of the colored configuration. 

\begin{lemma}\label{lemma-L-inf-lb3} The configuration consisting of black, red, and blue sites 
on $B$ is, with high probability, inert. 
\end{lemma}

\begin{proof}
Assume that, after the second stage, the configuration of the colored sites is not inert. Then
there is a non-red vertical line that contains $r$ black or blue sites within a vertical 
interval of size $2\rho+1$. Of these, at least one is blue (or the line would become red 
in the first stage) and therefore, by Lemma~\ref{lemma-L-inf-lb1},  
there are at least $s_2$ black sites within
a vertical interval of size $5r\rho$. 
So within such an interval, we can find some $i \in [1,r-s_2]$
blue sites and $r-i$ black sites. Assume also that there are
$j \in [0,s_1-1]$ red lines near this line. Each of the $i$ blue sites
has to have $s_1-j$ black sites on the horizontal line through it, nearby but not on the red lines; 
call these black sites, together with the $r-i$ black sites, {\it assisting\/}. Observe that we 
cannot have $j\le s_1-2$, as then again the line would be saturated at stage 1. Therefore, 
$j=s_1-1$.

Now we claim that none of the assisting black sites can be on the same horizontal line 
as a saturating site of one of the $j$ red lines. (Recall that different nearby
red lines do not have saturation sites on the same horizontal line by Lemma~\ref{lemma-L-inf-lb0}.) 
Indeed, if this were true, we would 
again have $a\ge 1$ horizontal lines that each cover an assisting and a saturating site. 
Each of these contains $2$ black sites, and additionally we have 
at least $2(r-a)$ black sites on the horizontal lines through saturation sites
and through assisting sites, and they are all nearby. This would again produce $2r$ nearby 
black sites, demonstrating that the black sites involved in saturating the lines and 
assisting sites are disjoint. 

It follows that we have $s_1-1$ nearby red lines that get saturated disjointly and, 
also disjointly, a nearby set of $r$ nearby black sites. The probability for 
this is at most a constant times  
$$
A(Ap^{-s_2}p^r)^{s_1-1} Ap^{-s_2}p^r\ll 1.
$$
Consequently, the probability that the configuration of colored sites 
is not inert converges to $0$.
\end{proof}

We have so far dealt with the dynamics on $B$, and obtained an inert configuration, 
which has low density, as we will observe in the proof of Lemma~\ref{lemma-L-inf-lb6}. 
However, this is not sufficient 
in this instance. The reason is that the box $B$ only has width $A$, and the simple ``speed
of light'' argument of Lemma~\ref{lemma-lb-easy} only gives a lower bound on the order 
of $A$, which is far smaller than the claimed size of $T$. We need to provide a different 
argument, a probabilistic one, that demonstrates that the influence from outside $B$ likely
spreads horizontally with speed no larger than on the order of $p^{s_2}$. 

To achieve this, we 
assume we have $A\times (A/p^{s_2})$ box $B$ with origin in its center,
and enlarge the initial configuration by replacing it with the occupied 
set of black, red, and blue sites 
produced by the two-stage procedure above. For this part, we will only need 
that every blue site has $s_2$ black sites within some
distance $R=R(\rho)$. We make all sites outside of the box occupied and call 
them {\it green\/}. Now run the original dynamics with the zero-set $\cZ$ on the resulting 
configuration and also call all newly occupied sites green. We claim that 
there exist an $\epsilon>0$ so that the origin turns green 
by time $\epsilon/p^{s_2}$ with probability approaching $0$. We call a sequence 
$x_0, x_1,..., x_n$ a {\it neighborhood path\/} if $x_{i}\in \cN_{x_{i-1}}$ for $i=1,\ldots, n$.
 
\begin{lemma}\label{lemma-L-inf-lb4} Assume that some site $x$ in $B$ turns green at time $t\ge 1$. 
Then there exists
a neighborhood path $x_0, x_1,..., x_n=x$ of length $t$, so that $x_0$ is
outside $B$ and for every horizontal step from $x_i$ to $x_{i+1}$ there are
$s_2$ black sites within distance $2R$ of $x_{i+1}$.
\end{lemma}

\begin{proof} We argue by induction. Due to inertness of other colors, $x$ must have a green
neighbor at time $t-1$. If that neighbor is vertical, the induction step
is trivial. Otherwise, $x$ must have a horizontal green neighbor, and also
$s_2$ non-red vertical neighbors (because $x$ itself is not red). If all of these 
vertical neighbors are black, the induction step is completed. If one of them is blue, it has 
$s_2$ black sites within distance $R$, which again completes the induction step.
\end{proof} 

The problem with the path in the previous lemma is that 
the same black sites may be used at many steps, so we 
need to construct a path that does not do that. For that, we abandon the 
``neighborhood'' assumption. We say that a path 
$y_0, ..., y_n=x$ makes a {\it lateral move\/}
from $y_i$ to $y_{i+1}$ if $||y_i-y_{i+1}||\le 5R$. Thus, a lateral move is allowed to be in
any direction within the restricted distance.

\begin{lemma} \label{lemma-L-inf-lb5} 
Let $x$ be as in
Lemma~\ref{lemma-L-inf-lb4}. Then there exists a path $y_0,\ldots, y_n=x$ of
length $n\le \rho t$ with the following properties:
\begin{itemize}
\item it is self-avoiding; 
\item it makes vertical neighborhood moves of
unit distance; and
\item if there is a
lateral move from $y_i$ to $y_{i+1}$, then there are $s_2$ black sites 
within distance $2R$ of $y_{i+1}$; and
\item all black sites
at different lateral moves are distinct.
\end{itemize} 
\end{lemma}


\begin{proof} Take the path $x_0,\ldots,x_n$ from the previous lemma 
let the associated set of black sites be $Z$. For every
$z\in Z$, let $X_z$ be the set of sites $x_i$ such that there is a 
horizontal step to $x_i$ which has $z$ as one of the requisite neighboring black sites. 
All elements of $X_z$ are within distance
$4R$ of each other. The $y$-path is constructed as follows. For every $X_z$, let 
$x_\ell$ be the the first site on the path in $X_z$  
and $x_m$ the last.  Now eliminate 
from the path all sites from $x_\ell$ up to, but not including, $x_m$.
The new path has steps bounded by $5R$: the step from $x_{\ell-1}$ to $x_\ell$ is 
bounded by $\rho$ and the distance
from $x_\ell$ to $x_m$ is bounded by $4R$. Moreover, any vertical segment of the path can be replaced
by a sequence of unit moves, all in the same direction, which lengthens the path by at
most a factor $\rho$.
\end{proof}

We now have everything in place for the last step in the proof of Theorem~\ref{thm-L-infinite}.

\begin{lemma}\label{lemma-L-inf-lb6} The lower bound $T\gtrsim p^{-\gamma}$ holds. Hence
$\gamma_c\ge \gamma$.
\end{lemma}

\begin{proof}
The probability that there exists a path from Lemma~\ref{lemma-L-inf-lb5} of length $s$ that makes 
$j$ lateral moves is at most
$$
\binom sj C^j p^{s_2j}\le \left(\frac {Cesp^{s_2}}{j}\right)^j,
$$
where $C=C(R)$ is a constant, the number of choices for a move at a lateral step (and those
are the only steps at which we have a choice) times the number of possible positions of
a black site at a lateral step.

Let $D$ be a large constant, $D>2Ce\rho$. Assume $s\le \rho t$. The probability that there
exists a path from Lemma~\ref{lemma-L-inf-lb5} of length $s$ that makes at 
least $Dtp^{s_2}$ lateral moves then is
at most
$$
\sum_{j\ge Dtp^{s_2}} \left(\frac {Cesp^{s_2}}{j}\right)^j \le \sum_{j\ge Dtp^{s_2}} \left(\frac {Ce\rho}{D}\right)^j\le 2 (0.5)^{Dtp^{s_2}}.
$$
Now the probability that there exists a path of length at most $\rho t$ that
makes at least $Dtp^{s_2}$ lateral moves is at most
$$
2\rho t (0.5)^{Dtp^{s_2}}.
$$
Now let $t=\epsilon Ap^{-s_2}$, where $\epsilon=1/(20 DR)$. A path from Lemma~\ref{lemma-L-inf-lb5}, 
with $x$ the origin,
would have to make at least $A/(10R) =2Dtp^{s_2}$ lateral moves. Note that
it is impossible to connect to initially
green sites through the top or the bottom of the box, due to the speed of light. Therefore, 
the probability that such a path exists is at most
$$
2\rho t (0.5)^{A/(20R)}\to 0.
$$
It follows that the origin is not green by time $t$ with high probability.

Finally, we verify that the origin is also unlikely to be any other color. 
First, we recall that the probability that the origin is black is $p$. Next, 
the probability that it is red is 
$\cO(Ap^{r-s_2})$ by definition of a red line. Finally, the probability that it is blue is
$o(1)$ by Lemma~\ref{lemma-L-inf-lb2}.
\end{proof}

\begin{proof}[Proof of Theorem~\ref{thm-L-infinite}]
The conclusion follows from Lemmas~\ref{lemma-L-inf-ub} and~\ref{lemma-L-inf-lb6}.
\end{proof}

\section{Open Problems}\label{sec-open}
\begin{enumerate}
\item Does the critical power $\gc$ exist for all zero-sets $\cZ$? 
\item For the symmetric line growth of Theorem~\ref{thm-line} with $r=s\ge 3$, is there  
a matching lower bound to the upper bound in Lemma~\ref{lemma-line-sym-upper}, i.e., is it true that 
$t\ll {p^{-(r-1)}}\log\frac1p$ implies that $\prob{T\le t}\to 0$?
\item For the line growth of Theorem~\ref{thm-line} with $1<s<r$, is the critical power pure?
\item For the L-shaped zero-set of part 3 of Theorem~\ref{thm-L-finite}, is $\gc=r-\cO(1)$ 
if $s=\Theta(r)$, for large $r$?
\item What is the critical power for the zero-set which is the union of the bootstrap percolation 
zero-set with threshold $r$, and an infinite strip of width $s<r$? Thus, in this rule, 
the condition for occupation is that the total number of occupied neighbors is at least $r$, 
and the number of vertical occupied neighbors is at least $s$. 
\item As in \cite{GSS}, consider the dynamics with zero-set $\cZ$ on a $M\times N$ box $B$, 
where $\log M/\log p\to \alpha$ and $\log N/\log p\to \beta$. 
Call $I=I(\alpha,\beta,\cZ)$ the {\it large deviation rate\/} if 
$\log\prob{\xi_\infty=B}/\log p\to I$. When can the large deviation rate be shown to exist and when 
can it be computed?

\end{enumerate}
 
\section*{Acknowledgments} 
JG was partially
supported by the Slovenian Research Agency research program
P1-0285 and  Simons Foundation Award \#709425. DS was partially supported by the NSF TRIPODS grant CCF--1740761. LW was supported by the Ohio 5-OSU Summer Undergraduate Research Experience (SURE) program.

\end{document}